\documentclass[11pt,a4paper]{article}

\usepackage[utf8]{inputenc}
\usepackage[T1]{fontenc}
\usepackage{amsmath, amsthm, amssymb, amsfonts}
\usepackage{geometry}
\usepackage{hyperref}
\usepackage{mathtools}
\usepackage{enumitem}
\usepackage{empheq}
\usepackage{authblk}
\usepackage[natbib,style=iso-numeric,sorting=none]{biblatex}
\usepackage{bm}
\newcommand{\bTheorem}[1]{
	\begin{Theorem} \label{T#1} }
	\newcommand{\eT}{\end{Theorem}}

\newcommand{\bProposition}[1]{
	\begin{Proposition} \label{P#1}}
	\newcommand{\eP}{\end{Proposition}}

\newcommand{\bLemma}[1]{
	\begin{Lemma} \label{L#1} }
	\newcommand{\eL}{\end{Lemma}}

\newcommand{\bCorollary}[1]{
	\begin{Corollary} \label{C#1} }
	\newcommand{\eC}{\end{Corollary}}

\newcommand{\bRemark}[1]{
	\begin{Remark} \label{R#1} }
	\newcommand{\eR}{\end{Remark}}

\newcommand{\bDefinition}[1]{
	\begin{Definition} \label{D#1} }
	\newcommand{\eD}{\end{Definition}}

\newcommand{\prst}{\mathbb{P}}
\newcommand{\Ds}{\mathbb{D}_x}

\newcommand{\bfphi}{\boldsymbol{\varphi}}

\newcommand{\lnorm}{\|}
\newcommand{\rnorm}{\|}
\newcommand{\Td}{\mathbb{T}^3}
\newcommand{\bFormula}[1]{
	\begin{equation} \label{#1}}
	\newcommand{\eF}{\end{equation}}

\newcommand{\vrn}{\vr_n}
\newcommand{\vun}{\vu_n}

\newcommand{\Ov}[1]{\overline{#1}}

\newcommand{\Curl}{{\bf curl}_x}

\newcommand{\aleq}{\stackrel{<}{\sim}}

\newcommand{\trn}{\widetilde{\vr}_n}
\newcommand{\tun}{\widetilde{\vu}_n}
\newcommand{\tBn}{\widetilde{\vc{B}}_n}
\newcommand{\tWni}{\widetilde{W}_{n}^1}
\newcommand{\tWnii}{\widetilde{W}_{n}^2}

\newcommand{\vr}{\varrho}

\newcommand{\vu}{\vc{u}}
\newcommand{\vm}{\vc{m}}

\newcommand{\vc}[1]{{\bf #1}}

\newcommand{\Div}{{\rm div}_x}
\newcommand{\Grad}{\nabla_x}

\newcommand{\dx}{\,{\rm d} {x}}

\newcommand{\dt}{\,{\rm d} t }

\newcommand{\intTd}[1]{\int_{\mathbb{T}^3} #1 \, \dx}

\newcommand{\D}{{\rm d}}

\newcommand{\R}{\mathbb{R}}

\newcommand{\expe}[1]{ \mathbb{E} \left[ #1 \right] }

\newcommand{\br}{ \nonumber \\ }

\newtheorem{theorem}{Theorem}[section]
\newtheorem{definition}[theorem]{Definition}
\newtheorem{lemma}[theorem]{Lemma}
\newtheorem{corollary}[theorem]{Corollary}
\newtheorem{remark}[theorem]{Remark}

\title{\textbf{On the existence of a weak martingale solution for a stochastic magnetohydrodynamics system with noise acting in the magnetic field}}
\author[1,2]{Jiří Púček\thanks{The work of J. P. was supported by the Czech Science Foundation (GAČR), Grant Agreement No. 24-11034S. The Institute of Mathematics of the Academy of Sciences of the Czech Republic is supported by RVO:67985840.}}

\affil[1]{\raggedright \hangindent=0.5em
\hangafter=1 Charles University, Faculty of Mathematics and Physics, Mathematical Institute, Sokolovská 83, 186 75 Prague 8, Czech Republic.}
\affil[2]{\raggedright \hangindent=0.5em
\hangafter=1 Institute of Mathematics of the Academy of Sciences of the Czech Republic, Žitná 25, CZ-11567 Prague 1, Czech Republic.}
\date{}

\begin{document}

\maketitle

% --- ABSTRACT ---
\begin{abstract}
We prove the existence of solutions to the stochastic  magnetohydrodynamics (MHD) system, where randomness is introduced through random initial data and a stochastic integral appearing solely in the induction equation, while the fluid equations remain deterministic. We focus on the case where the adiabatic exponent $\gamma$ satisfies $\gamma > \frac{3}{2}$. The existence proof is carried out using the penalization method. We define the notion of a martingale solution and establish sufficient conditions for its existence. The proof then proceeds by means of the stochastic compactness method. Using an energy inequality, we derive a priori estimates in terms of expectation. Due to the stochastic nature of the problem, we demonstrate convergence in law of the penalized solutions and verify that the limiting object is a martingale solution.

\end{abstract}

% --- INTRODUCTION ---
\section{Introduction}
The study of magnetohydrodynamics (MHD) plays a fundamental role in understanding physical phenomena in plasma physics, astrophysics, and engineering applications such as fusion reactors (see, e.g., \cite{davidson2001mhd}). However, in many real-world settings, such systems are not purely deterministic. External noise and measurement uncertainty (we refer to \cite{taylor1997error}) motivate the incorporation of random perturbations into the governing equations.

The main motivation behind the model considered in the study is the description of the star, where the source term is supported in the stellar core. Unfortunately, in actual models, it is not exactly known how the processes in the star core work since stellar cores cannot be observed directly (see, e.g.,~\cite{clayton1983stellar}). In such contexts, a stochastic formulation of the PDE system can become useful.

Motivated by this setting, we investigate a three-dimensional stochastic MHD system in which randomness is introduced through random initial data and a stochastic integral appearing solely in the induction equation. In contrast, the continuity and momentum equations remain deterministic. Specifically, we study a compressible Navier--Stokes system immersed in a magnetic field that evolves under stochastic noise. The magnetic field is defined on a larger periodic domain $\mathbb{T}^3$, while the fluid occupies only a subdomain $Q_F$.

To date, relatively little work has been done on problems of this type. Similar problems have recently been studied by H. Wang in \cite{WangH21} and in the follow-up work \cite{WangH}. The first of these works contains issues related to boundary conditions and incompatibility with solenoidal constraints. We address these shortcomings by formulating the stochastic forcing in the induction equation in a way that preserves the divergence-free property of the magnetic field and is compatible with periodic boundary conditions. Following Wang's work \cite{WangH21}, we therefore consider a setting in which the fluid equations (i.e. the continuity and momentum) are defined on the periodic domain $\Td$, accompanied by a friction term in the momentum equation. This term will serve as a penalization outside the fluid domain $Q_F$ and will ensure that the fluid density and velocity vanish outside this region in the limit (see Figure \ref{fig:figure1} in Section \ref{k}). Unlike the deterministic case, where convergence in space and time is sufficient, the stochastic nature of our system requires convergence in law as well. For this reason, we employ the stochastic compactness method to establish the convergence.

The main mathematical objective is to establish the existence of a weak martingale solution to the MHD system under suitable assumptions on the initial data for an adiabatic exponent $\gamma > \frac{3}{2}$. The key novelties include:
\begin{itemize}[itemsep=2pt]
  \item The fluid is immersed in the magnetic field, whereas most existing results assume that both are defined on the same domain.
  \item The source term for the induction equation is a random process and acts outside the fluid domain (see Section \ref{k})
  \item Domain penalization is used for a stochastic PDE system.
  \item The stochastic forcing is formulated to preserve the divergence-free constraint on the magnetic field.
\end{itemize}

The paper is organized as follows: in Section \ref{k} we present the classical formulation and motivate the introduction of randomness in the system, leading to the stochastic PDE. In Section \ref{sec:mainresults} we define the weak martingale solution of our problem, as well as the penalized problem and present the main results of this paper. Finally, in the last Section \ref{proof}, we provide the proof.

\section{Classical formulation of the problem}
\label{k}
\begin{figure}[h!]
	\begin{center}
		\includegraphics[scale=0.37]{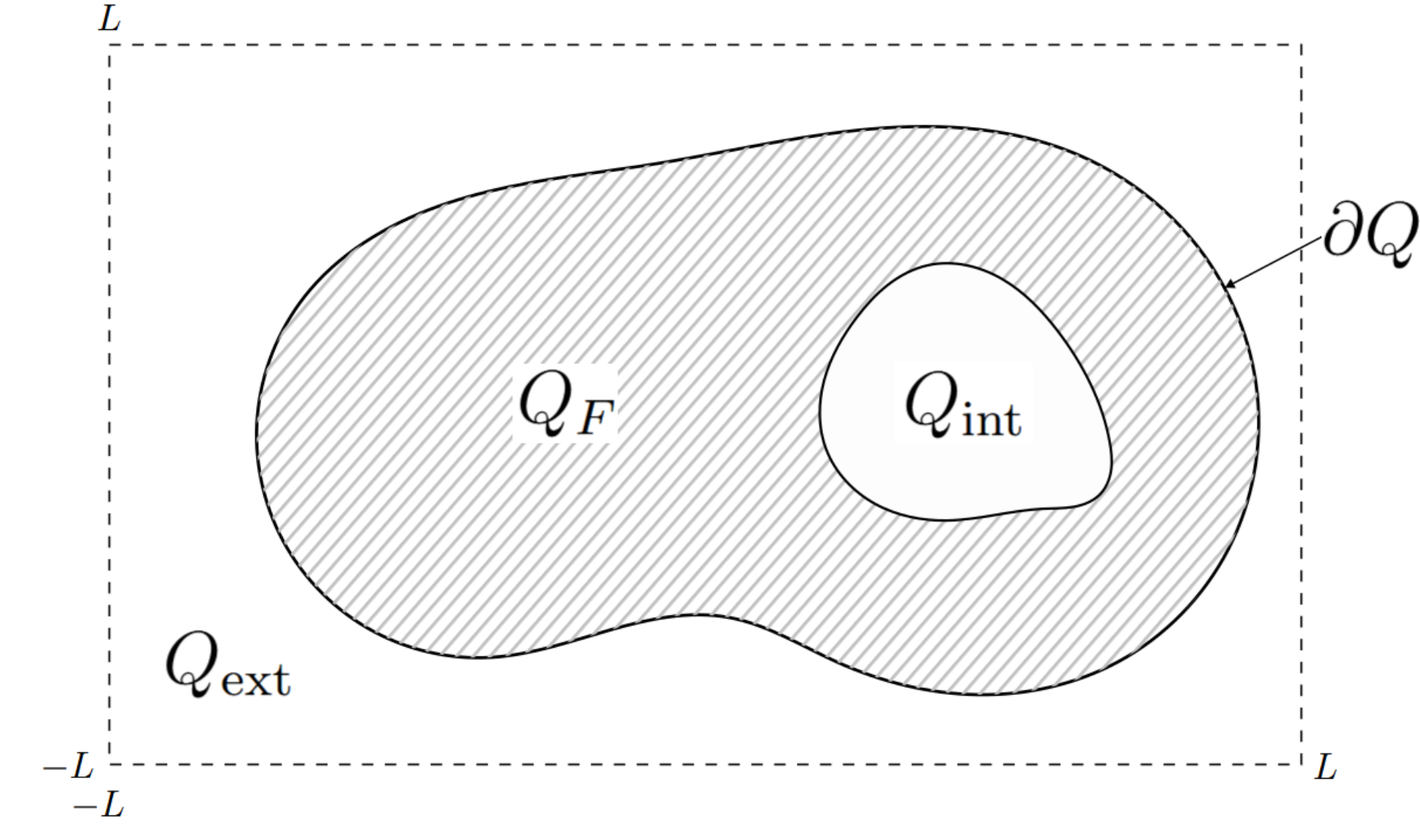}
	\end{center}
    \caption{\label{fig:figure1} Sketch of the domains $Q_{\text{ext}}$, $Q_F$, $Q_{\text{int}}$}
\end{figure}
We consider $Q \subset \R^3$ a bounded domain of class $C^{2+\xi}$ for $\xi>0$ and suppose that
\begin{equation} \label{g1}
Q \subset \Ov{Q} \subset \mathbb{T}^3,\ 
\mathbb{T}^3 = \left( [-L, L] \Big|_{\{ -L; L \}} \right)^3, \ L>0.
\end{equation}

In addition, we consider another domain 
$Q_{\rm int}$, again of class $C^{2+\xi}$, such that
\begin{equation} \label{g2}
Q_{\rm int} \subset \Ov{Q}_{\rm int} \subset Q. 
\end{equation}
It should be pointed out that the class of domains $Q$ and $Q_{\rm int}$ can be relaxed to the Lipschitz class. This fact will be addressed after presenting the main result.
\\

As can be seen in Figure \ref{fig:figure1}, we denote 
\begin{equation} \label{g3}
Q_{\rm ext} = \Td \setminus \Ov{Q},\ 
Q_F = Q \setminus \Ov{Q}_{\rm int}. 
\end{equation}	

The fluid will occupy the domain $Q_F$, and the magnetic field will be active in the whole space $\Td$. One can interpret the domain $Q$ as a star in a vacuum, and the domain $Q_{\rm int}$ as the star core, where the source term is supported. The unknowns are:

$\vr = \vr(t,x), t \in [0,T], x \in Q_F$ \dotfill 
fluid mass density 

$\vu = \vu(t,x), t \in [0,T], x \in Q_F$ \dotfill fluid velocity 

$\vc{B} = \vc{B}(t,x), t \in [0,T], x \in \Td$ \dotfill 
magnetic field
\\

For the fluid equation, we have the standard equation of continuity and the momentum equation that, in stochastic notation, reads
\begin{equation} \label{f1f}
\D \vr + \Div (\vr \vu) \dt = 0 \ \mbox{in}\ (0,T) \times Q_F,
\end{equation}
\begin{equation}\label{f2f}
    \D (\vr \vu) + \Div (\vr \vu \otimes \vu) \dt + \Grad p \dt = \Div \mathbb{S} \dt +\vr \Grad G \dt + \vc{f} \dt \ \mbox{in}\ (0,T) \times Q_F.
\end{equation}

Throughout the proof, we have to ensure that the density is at least non-negative. In fact, the strict positivity of $\vr$ in the context of weak solutions is a major open problem.

The \emph{pressure} $p$ is given by the equation of state 
\begin{equation} \label{f3}
	p =p(\vr)=a\vr^\gamma, \ a>0.
\end{equation}
Without loss of generality, we shall choose the constant $a=1$. The physically relevant values of $\gamma$ for the isentropic gas are in the range $1 < \gamma \leq\frac{5}{3}$, where the case $\gamma=\frac{5}{3}$ corresponds to the plasma (we refer to \cite{Battaner1996}). The best theoretical results for the isentropic system guarantee the global existence of weak solutions for $\gamma > \frac{d}{2}$, where $d$ is the dimension (see, for example, \cite{FeireislEduard2003DoVC}). Therefore, in three dimensions, we assume throughout this paper $\gamma>\frac{3}{2}$.

We consider Newtonian fluids with the \emph{viscous stress 
tensor} 
\begin{equation*}
\mathbb{S} = \mathbb{S}(\Ds \vu) = \nu \left( \Grad \vu + \Grad^t \vu - \frac{2}{3} \Div \vu \mathbb{I} \right) + \lambda \Div \vu \mathbb{I}, \ 
\Ds \vu = \frac{1}{2} \left( \Grad \vu + \Grad^t \vu \right),
	\end{equation*}
where $\nu > 0$ is the \emph{shear viscosity coefficient}, and $\lambda \geq 0$ 
is the \emph{bulk viscosity coefficient}.

The external \emph{volume force} $\vc{g}$ is prescribed. We consider the gravitational force
\begin{equation*}
	\vc{g} = \Grad G,\ G = G(x) \in W^{1,\infty}(Q_F), 
\end{equation*}	
where $G$ is the gravitational potential. 

If the fluid is electrically conducting, the driving force $\vc{f}$ is the \emph{Lorentz force}
\begin{equation} \label{f6}
\vc{f} =  \Curl \vc{B} \times \vc{B}.
\end{equation}

Regarding the boundary conditions, we impose the no-slip condition on the velocity $\vu$
\begin{equation} \label{f7}
	\vu|_{\partial Q_F} = 0. 
	\end{equation}
	
In accordance with \eqref{f7}, we extend the velocity to be zero outside $Q_F$. In particular, we set 
\begin{equation} \label{f8} 
	\vu|_{(0,T) \times \Td\setminus Q_F} = 0.
\end{equation}

The time evolution of the magnetic field $\vc{B}$ is determined by Maxwell's equation
\begin{equation} \label{m1}
	\partial_t \vc{B} + \Curl \vc{E} = 0,\ \Div \vc{B} = 0 \ \mbox{in}\ (0,T) 
	\times \Td,  
\end{equation}	
where the electric field $\vc{E}$ obeys the \emph{Ohm--Ampère law},
\begin{equation} \label{m2}
	\Curl \vc{B} = \frac{1}{\eta} \left( \vc{E} + (\vu \times \vc{B} \right) + \vc{j}), 
\end{equation}
where $\vc{B}$ is the magnetic field, and $\vc{j}$ is a given electric current.

The resistivity of the material $\eta$ generally varies across the subdomains
\begin{equation} \label{m3}
	\eta (x) = \left\{ \begin{array}{l}  \eta_{\rm int} > 0 \ \mbox{for}\ 
		x \in Q_{\rm int}, \\ 
		\eta_{F} > 0 \ \mbox{for}\ 
		x \in Q_{F},\\  
		\eta_{\rm ext} > 0 \ \mbox{for}\ 
		x \in Q_{\rm ext}.
	\end{array}	\right.	\ \mbox{deterministic constants}
\end{equation}
Therefore, $\eta$ is a positive piecewise constant function.

Observe that in the whole model, we made a simplification and considered permeability to be equal to one; thus, this parameter does not appear in the induction equation or in the Lorentz force \eqref{f6}.

Accordingly, equations \eqref{m1} and \eqref{m2} can be written in a concise form 
\begin{equation} \label{m5}
	\partial_t \vc{B} + \Curl (\vc{B} \times \vu) + \Curl \Big( \eta\,
	\Curl \vc{B} \Big) = \Curl \vc{j} \ \mbox{in}\ (0,T) \times \Td. 
\end{equation}	
	
The driving force in \eqref{m5} is represented by $\Curl \vc{j}$. However, in view of our interpretation of a star, the precise physical mechanism occurring in the stellar core is not known, and any measurement related to the core would naturally be subject to uncertainty, motivating the introduction of a stochastic forcing term. Therefore, our goal is to consider the random driving force in the form of a stochastic integral, which represents the noise in the system. In particular, we consider the driving force to be in the form
\begin{equation*}
\alpha \vc{B} \D W^1 + \Curl \vc{j} \D W^2,
\end{equation*}
where $W^1$, $W^2$ are independent one-dimensional Wiener processes, $\alpha$ is a real deterministic constant, and $\vc{j}$ is a random process such that $\prst$-a.s.
\begin{equation} \label{f5}
\alpha \in \R,\ \vc{j} \in L^\infty(0,T; W^{1,2}_0 (Q_{\rm int}; \R^3)) \ \mbox{such that} \ \mathbb{E}\| \Curl\vc{j}\|^{2r}_{L^\infty(0,T;L^2(Q_{\rm int};\R^3))}<\infty,
\end{equation}

This gives us the induction equation in the stochastic setting
\begin{align} 
\D \vc{B} + \Curl (\vc{B} \times \vu) \dt + \Curl (\eta \, \Curl \vc{B}) \dt &= \alpha \vc{B} \D W^1 + \Curl \vc{j} \D W^2, \br 
\Div \vc{B} &= 0.
\label{f4f}
\end{align}
Note that in our case, it is compatible with the solenoidal condition as long as $\alpha$ is constant. Indeed, formally by applying divergence to \eqref{f4f}, we get 
\[
\D (\Div \vc{B}) = \alpha (\Div \vc{B}) \, \D W^1,\ \Div \vc{B}(0, \cdot) = \Div \vc{B}_0,
\]
which implies
\begin{equation*}
	\Div \vc{B} (t, \cdot) = 0 \ \mbox{as long as}\ \Div \vc{B}_0 = 0.
\end{equation*}

Thus, our MHD system \eqref{f1f}, \eqref{f2f}, \eqref{f4f} of the stochastic partial differential equation is
\begin{equation}
\boxed{
\begin{aligned}
\D \vr + \Div (\vr \vu) \dt  &= 0, 
&& \text{in } (0,T)\times Q_F, \\
\D (\vr \vu) + \Div (\vr \vu \otimes \vu) \dt + \Grad p \dt 
&= \Div \mathbb{S} \dt +\vr \Grad G \dt + \vc{f} \dt,
&& \text{in } (0,T)\times Q_F, \\
\D \vc{B} + \Curl (\vc{B} \times \vu) \dt + \Curl (\eta \, \Curl \vc{B}) \dt  
&= \alpha \vc{B} \D W^1 + \Curl \vc{j} \D W^2, 
&& \text{in } (0,T)\times \mathbb{T}^3, \\
\Div \vc{B} &= 0,
&& \text{in } (0,T)\times \mathbb{T}^3.
\end{aligned}
}
\label{eq:MHD}
\end{equation}
with boundary condition \eqref{f7} (or, more precisely, with condition \eqref{f8}) and initial conditions that will be discussed later.

\subsection{Energy estimates}
The energy balance is (formally) obtained by applying 
It\^o's formula to the induction equation:

\begin{align}
	\frac{1}{2} \D |\vc{B}|^2  &+  
	\Curl (\vc{B} \times \vu ) \cdot \vc{B}  \dt  + 
	\Curl (\eta \Curl \vc{B} ) \cdot \vc{B} \dt \br
	&= \frac{\alpha^2}{2} |\vc{B}|^2  \dt 
	+ \frac{1}{2} |\Curl \vc{j}|^2 \dt 
	+ \alpha |\vc{B}|^2(t, \cdot) \D W^1 + \ \Curl \vc{j} \cdot 
	\vc{B} \D W^2.   
	\label{f9f}
\end{align}

Similarly, we multiply the momentum equation by $\vu$ to obtain the balance of kinetic energy. Since there is no stochastic integral, the energy balance is standard: 
\begin{align} 
\D E(\vr, \vr \vu) + \Div \Big[  (E(\vr, \vr \vu) + p) \vu \Big] \dt 
 =& \Div (\mathbb{S} \cdot \vu) \dt -\mathbb{S}: \Grad \vu \dt\br
& +(\Curl \vc{B} \times \vc{B}) \cdot \vu \dt\br
& +\vr \vu \cdot \Grad G \dt,
\label{f10}
\end{align}	
where the energy is defined as
\[
E(\varrho, \mathbf{m}) =
\begin{cases}
\frac{1}{2}\frac{|\mathbf{m}|^2}{\varrho} + P(\varrho)
& \text{if } \varrho > 0, \\[0.5em]
P(\varrho)
& \text{if } \varrho = 0,\ \mathbf{m} = 0, \\[0.5em]
\infty & \text{otherwise,}
\end{cases}
\]
where $P=P(\vr)$ is the pressure potential that satisfies
\[
P'(\vr) \vr - P(\vr) = p(\vr), 
\]
especially for our case, we shall see that
\begin{equation*}
    P(\vr)=\frac{1}{\gamma-1}\vr^\gamma.
\end{equation*}

The energy $E: \R^7 \to [0, \infty] $ is a convex lower semi-continuous function. Summing up \eqref{f9f} and \eqref{f10}, we get
\begin{align}\label{energ_nerovnost}
	& \D \left(E+\frac{1}{2}|\vc{B}|^2\right)  + 
	\Curl (\eta \Curl \vc{B} ) \cdot \vc{B} \dt + \Div \Big[  (E(\vr, \vr \vu) + p) \vu -\mathbb{S} \cdot \vu\Big] \dt \br
	&= \frac{1}{2}\left( \alpha^2|\vc{B}|^2 +|\Curl \vc{j}|^2 \right)\dt 
	-\mathbb{S}: \Grad \vu \dt 
    + \alpha |\vc{B}|^2 \D W^1 + \ \Curl \vc{j} \cdot 
	\vc{B} \D W^2.   
\end{align}
The energy balance plays a crucial role, as it provides a priori estimates and will be incorporated into the definition of the weak martingale solution to be defined in the next section. Since uniqueness of solutions is not expected in this framework, incorporating the energy balance into the solution concept helps compensate for the lack of uniqueness.
% --- MAIN RESULTS ---
\section{Main Results}
\label{sec:mainresults}
Suppose that we are given initial data 
\begin{align} 
\vr_0 &\in L^1 (Q_F),\ \vr_0 \geq 0,\  0 < \underline{\vr} \leq \int_{Q_F} 
\vr_0 \ \dx \leq \Ov{\vr},\ \underline{\vr}, \Ov{\vr} \ \ \mbox{deterministic constants} \br 
(\vr \vu)_0 &\in L^1(Q_F), \
\expe{ \left( \int_{Q_F} \left[ \frac{1}{2} \frac{|(\vr \vu)_0|^2}{\vr_0} + \frac{1}{\gamma -1} \vr^\gamma \right] \dx \right)^r } 	 \aleq 1, 
 \br 
\vc{B}_0 &\in L^2(\Td; \R^3),\ \expe{ \left( \intTd{ |\vc{B}_{0} |^2 } \right)^r } \aleq 1, \ \text{for some } r \geq 4,\br
\int_{\Td} \vc{B}_0& \cdot\Grad\psi \ \dx = 0,\ \psi \in C^1(\Td),	
\label{w1}
\end{align}
where the initial data are assumed to be random. Here, $\aleq 1$ denotes $\leq C$ for some constant $C$ depending on the data but not on $n$. Observe that the initial density and momentum of the fluid are considered only in the domain $Q_F$, since no fluid is present outside this domain. In contrast, the initial data of the magnetic field is defined on the whole torus $\Td$.

\begin{definition}[{\bf Weak martingale solution}] \label{wD1}
A weak martingale solution of problem \eqref{eq:MHD} and \eqref{energ_nerovnost} is the quantity 
\[
\mathcal{U} = \Big( (\Omega, \mathcal{A}, (\mathcal{A}_t)_{t \geq 0}, 
\prst ) , \vr, \vu, \vc{B}, W^1, W^2 \Big), 
\] 	
where 
\begin{enumerate}
	\item 
	$(\Omega, \mathcal{A}, (\mathcal{A}_t)_{t \geq 0}, 
	\prst )$ is a complete stochastic basis with a right--continuous filtration $(\mathcal{A}_t)_{t \geq 0}$.
	\item $(W^1, W^2)$ are independent  $(\mathcal{A}_t)_{t \geq 0}$ Wiener processes. 
\item 
\begin{align} 
\vr &\in C_{\rm weak}([0,T]; L^\gamma (Q_F)), \br 
\vr \vu &\in C_{\rm weak}([0,T]; L^{\frac{2 \gamma}{\gamma + 1}}(Q_F; \R^3)), \br 
\vc{B} &\in C_{\rm weak}([0,T]; L^2(\Td; \R^3)) 	
\nonumber
\end{align}
are $(\mathcal{A}_t)$ progressively measurable stochastic processes, 
\[
\vu \in L^2(0,T; W^{1,2}_0(Q_F; \R^3)) 
\]
is an $(\mathcal{A}_t)$ adapted random distribution, where the proper definition of random distribution has been given in \cite{BrFeHobook}, Chapter 2.2.
\item The joint law of the initial data
\begin{align}
[\vr(0, \cdot), (\vr\vu)(0, \cdot), \vc{B}(0,\cdot)]  &\sim [\vr_0, (\vr \vu)_0, \vc{B}_0] \ \mbox{(equivalence in law)} \br &\quad \mbox{in}\ L^1 (Q_F) 
\times L^1 (Q_F; \R^3) 
\times L^1(\Td; \R^3).  
\nonumber  
\nonumber
\end{align}

\item The equation of continuity is satisfied in the weak sense $\prst$-a.s.: 
\begin{equation} \label{w2} 
\left[ \int_{Q_F} \vr \varphi \ dx \right]_{t=0}^{t = \tau} 
= \int_0^\tau \int_{Q_F} \Big[ \vr \partial_t \varphi + \vr \vu \cdot 
\Grad \varphi \Big] \dx \dt 
\end{equation}	
for any $\varphi \in C^1([0,T] \times \Ov{Q}_F)$. 	
\item The renormalized equation of continuity is satisfied in the weak sense $\prst$-a.s.: 
\begin{equation}\label{w2re}
   \left[ \int_{Q_F}{ b(\vr) \varphi \dx} \right]_{t=0}^{t = \tau}= \int_0^\tau\int_{Q_F}[b(\vr) \partial_t \varphi + b(\vr) \vu \cdot \Grad \varphi+(b(\vr)-b'(\vr)\vr)\Div \vu\varphi]\D x\,  \D t, 
\end{equation}
for any $\varphi \in C^1([0,T] \times \Ov{Q}_F)$ and for any $b(z)$ such that $b \in C^1([0,\infty)), b' \in C_c([0,\infty))$. 
\item
The momentum equation is satisfied in the weak sense $\prst$-a.s.: 
\begin{align}
    \left[ \int_{Q_F}{ \vr \vu \cdot \bfphi } \right]_{t = 0}^{t = \tau}= &\int_0^\tau\int_{Q_F}[\vr \vu\cdot \partial_t \bm{\varphi} + (\vr \vu\otimes \vu): \Grad \bm{\varphi}+p(\vr)\Div \bm{\varphi} \br
    &  -\nu \Grad \vu : \Grad \bm{\varphi}-(\mu+\nu) \Div \vu\, \Div \bm{\varphi}-\vr \Grad G \cdot \bm{\varphi}]\D x\, \D t \br
    & +\int_0^\tau\int_{Q_F}[\Curl \vc{B}\times\vc{B}] \cdot \bm{\varphi}\, \D x\, \D t 	
\label{w3}
\end{align}	
for any $\bfphi \in C^1_c([0,T] \times Q_F; \R^3)$.
\item 
The induction equation is satisfied in the weak sense $\prst$-a.s.:
\begin{align} 
\left[ \intTd{ \vc{B} \cdot  \bm{\Psi} } \right]_{t = 0}^{t = \tau} &=\int_0^\tau\intTd{\vc{B} \cdot \partial_t \bm{\Psi} }\dt\br
 &-\int_0^\tau\int_{\Td}[  \eta \, \Curl \vc{B}\cdot \Curl\Psi
 +(\vc{B}\times\vu )\cdot \Curl\Psi]\D x \, \D t \br
    &+\int_0^\tau \left(\int_{\Td}\alpha\vc{B} \cdot \bm{\Psi}\, \D x\right) \D W^1	 \br
 &+\int_0^\tau \left(\int_{\Td}  \Curl\vc{j} \cdot\bm{\Psi} \, \D x\right) \D W^2
\label{w4}
\end{align}	
for any $\bm{\Psi} \in C^1([0,T] \times \Td, \R^3)$, where function $\eta$ is the deterministic piecewise constant function defined in \eqref{m3} and
\begin{equation} \label{w5}
\intTd{ \vc{B}(\tau, \cdot) \cdot \Grad \psi } = 0 
\end{equation} 
$\prst$-a.s. for any $\psi \in C^1(\Td)$.

\item The energy inequality
\begin{align}
   &\left[ \int_{Q_F} E(\vr, \vr \vu)  \dx + 
\intTd{ \frac{1}{2} |\vc{B}|^2  }\right]^{t=\tau}_{t=0} -\int_0^\tau \intTd{\vr\vu\cdot \Grad G}\br &+ 
\int_0^\tau \int_{Q_F} \mathbb{S}(\Grad \vu) : \Grad \vu \ \dx
+ \int_0^\tau \intTd{ \eta |\Curl \vc{B} |^2 } \dt  \br 
&\leq\frac{\alpha^2}{2} \int_0^\tau \intTd{ |\vc{B}|^2(t, \cdot) } \dt 
+ \frac{1}{2}\int_0^\tau \intTd{ |\Curl \vc{j}|^2 } \dt \br &+ \int_0^\tau \left( \intTd{ \alpha |\vc{B}|^2(t, \cdot) } \right) \D W^1 + \int_0^\tau \left( \intTd{  \Curl \vc{j} \cdot 
\vc{B}(t, \cdot)} \right) \D W^2
\label{w6}
\end{align}	
holds $\prst$-a.s.
	
\end{enumerate}	
It is important to note that all test functions are taken to be deterministic.
\end{definition}

Let us point out that, similarly to the case of the classical solution, the relation \eqref{w5} holds if we assume
\begin{equation}\label{w7}
    \int_{\Td} \vc{B}(0,\cdot) \cdot \Grad\psi \, \D x=0,
\end{equation}

by choosing $\bm{\Psi}= \Grad \psi$ in \eqref{w4}. Proving the existence of a weak martingale solution in the sense of Definition \ref{wD1} is the main goal of this paper. For this reason, we shall also define our problem on the whole torus using the penalization method similar to \cite{BaFeLMMiYu} (recall that in \cite{BaFeLMMiYu} we have a different system, but the procedure is the same).

\subsection{Problem on (periodic) torus}
Here, all equations are satisfied in $(0,T) \times \Td$, meaning that all quantities are spatially periodic. As before, neither the equation of continuity nor the momentum equation contains stochastic forcing. The key difference is the presence of a ``friction term'' $\beta \vu$ in the momentum equation \eqref{f2f}. Specifically, the momentum equation takes the form
\begin{equation*}
    \D (\vr \vu) + \Div(\vr \vu \otimes \vu)\dt+\Grad p(\vr) \dt= \Div \mathbb{S}(\mathbb{D}\vu) \dt+\vr \Grad G \dt+ \Curl \vc{B}\times \vc{B} \dt - \beta \vu \dt.
\end{equation*}

For the time being, we suppose 
\begin{equation} \label{S5}
\beta \in L^\infty (\Td),\ \beta \geq 0 
\end{equation}
is a given deterministic function. Later, it will be used as penalization outside the fluid domain $Q_F$, where the density and the fluid velocity are zero. The martingale solutions are defined in a similar way to Definition \ref{wD1}. Therefore, we consider the initial data
\begin{align} 
	\vr_0 &\in L^1 (\Td),\ \vr_0 \geq 0,\  0 < \underline{\vr} \leq \intTd{ 
	\vr_0}  \leq \Ov{\vr},\ \underline{\vr}, \Ov{\vr} \ \ \mbox{deterministic constants} \br 
	(\vr \vu)_0 &\in L^1(\Td; \R^3), \
	\expe{ \left( \intTd{ \left[ \frac{1}{2} \frac{|(\vr \vu)_0|^2}{\vr_0} + \frac{1}{\gamma -1} \vr^\gamma \right] } \right)^r } 	 \aleq 1,\ 
	\br 
	\vc{B}_0 &\in L^1(\Td; \R^3),\ \expe{ \left( \intTd{ |\vc{B}_{0} |^2 } \right)^r } \aleq 1, \ r \geq 4,\br
    \int_{\Td} &\vc{B}_0 \cdot \Grad \psi\, \dx = 0,\ \psi \in C^1(\Td).	\label{P1}
\end{align}

The difference between \eqref{w1} and \eqref{P1} is that now all the functions are defined on the whole torus $\Td$. This problem leads to a very similar definition of a weak martingale solution as in Definition \ref{wD1}.
\begin{definition}[{\bf Weak martingale solution on the torus}] \label{PD1}
    
	A weak martingale solution on $(0,T) \times \Td$ is the quantity 
	\[
	\mathcal{U} = \Big( (\Omega, \mathcal{A}, (\mathcal{A}_t)_{t \geq 0}, 
	\prst ) , \vr, \vu, \vc{B}, W^1, W^2 \Big), 
	\] 	
	where 
	\begin{enumerate}
		\item 
		$(\Omega, \mathcal{A}, (\mathcal{A}_t)_{t \geq 0}, 
		\prst )$ is a complete stochastic basis with a right--continuous filtration $(\mathcal{A}_t)_{t \geq 0}$.
		\item $(W^1, W^2)$ are independent  $(\mathcal{A}_t)_{t \geq 0}$ Wiener processes. 
		\item 
		\begin{align} 
			\vr &\in C_{\rm weak}([0,T]; L^\gamma (\Td)), \br 
			\vr \vu &\in C_{\rm weak}([0,T]; L^{\frac{2 \gamma}{\gamma + 1}}(\Td; \R^3)), \br 
			\vc{B} &\in C_{\rm weak}([0,T]; L^2(\Td; \R^3)) 	
			\nonumber
		\end{align}
		are $(\mathcal{A}_t)$ progressively measurable stochastic processes, 
		\[
		\vu \in L^2(0,T; W^{1,2}(\Td; \R^3)) 
		\]
		is $(\mathcal{A}_t)$ adapted random distribution.
		\item The joint law of the initial data
		\begin{align}
			[\vr(0, \cdot), (\vr \vu)(0, \cdot), \vc{B}(0,\cdot)]  &\sim [\vr_0, (\vr\vu)_0, \vc{B}_0] \ \mbox{(equivalence in law)} \br &\quad \mbox{in}\ L^1 (\Td) 
			\times L^1 (\Td; \R^3) 
			\times L^1(\Td; \R^3).  
			\nonumber
		\end{align}
		
		\item The equations \eqref{w2}, \eqref{w2re}, \eqref{w3} are satisfied in the weak sense $\prst$-a.s. on the whole domain $\Td$ instead of $Q_F$. Namely, we now consider energy inequality of type
        \begin{align}
			&\left[ \int_{\Td} E(\vr, \vr \vu)  \dx + 
\frac{1}{2}\intTd{ |\vc{B}|^2  } \right]^{t=\tau}_{t=0} + \int_0^\tau \intTd{ \beta |\vu|^2 } \dt  \br&+ 
			\int_0^\tau \intTd{ \mathbb{S}(\Grad \vu) : \Grad \vu }
			+ \int_0^\tau \intTd{ \eta |\Curl \vc{B} |^2 } \dt 
			\br 
			&\leq\frac{\alpha^2}{2} \int_0^\tau \intTd{ |\vc{B}|^2(t, \cdot) } \dt 
+ \frac{1}{2}\int_0^\tau \intTd{ |\Curl \vc{j}|^2 } \dt \br &+ \int_0^\tau \left( \intTd{ \alpha |\vc{B}|^2(t, \cdot) } \right) \D W^1 + \int_0^\tau \left( \intTd{  \Curl \vc{j} \cdot 
\vc{B}(t, \cdot)} \right) \D W^2 
			\label{P6}
		\end{align}	
		holds $\prst$-a.s.
        \\
        
        Equations \eqref{w4} and \eqref{w5} hold in exactly the same way as in Definition \ref{wD1} with $\eta \in C^\infty (\Td)$.
	\end{enumerate}		
\end{definition}	
By \cite{WangH21}, under conditions \eqref{P1}, there exists a weak martingale solution in the sense of Definition \ref{PD1} in our setting; i.e., the periodic boundary conditions and stochastic terms fulfill the solenoidal condition.

\subsection{Penalization}

Our aim is to construct solutions that satisfy the fluid equations only in $Q_F$, starting from solutions on $\Td$ from Definition \ref{PD1}. 
To this end, we consider penalization of Brinkman's type 
\begin{equation} \label{M2}
\beta_n = n \bm{1}_{\Td \setminus \Ov{Q}_F} \in L^\infty(\Td). 
\end{equation}
In addition, we consider a sequence of smooth approximations $\eta_n$, 
\begin{align} 
\eta_n &\in C^\infty(\Td),\ 
0<\min \{\eta_F, \eta_{\rm int}, \eta_{\rm ext} \} \leq \eta_n \leq 
\max \{\eta_F, \eta_{\rm int}, \eta_{\rm ext} \}, \br 
\eta_n(x) &\to \eta(x) = \left\{ \begin{array}{l} 
\eta_{\rm int}\ \mbox{for}\ x \in \Ov{Q}_{\rm int}, \\ 
\eta_F \ \mbox{for}\ x \in Q_F, \\ 
\eta_{\rm ext} \ \mbox{for}\ x \in \Ov{Q}_{\rm ext} \end{array} \right.
\ \mbox{for all}\ x \in \Td 	 \label{M3} 
\end{align} 
Note that such a sequence exists, since $\eta \in L^\infty(\Td)$ is a piecewise constant function.
\\

The main result proved in this paper is the existence of the weak martingale solution in the sense of Definition \ref{wD1} as formulated in the following theorem.

\begin{theorem}[\bf Convergence of penalized solutions.] \label{MT1}

Let $\gamma > \frac{3}{2}$.	
Suppose that the torus $\Td$ admits a decomposition into $Q_{\rm int}$, 
$Q_F$, and $Q_{\rm ext}$ specified in \eqref{g1}--\eqref{g3}, where $Q$ and
$Q_F$ are of class $C^{2+\xi}$ for $\xi>0$. Let the (random) initial data
$(\vr_0, \vm_0, \vc{B}_0)$ satisfy \eqref{w1} for some $r \geq 4$.
Let $\beta = \beta_n$, $\eta = \eta_n$ be as in \eqref{M2}, \eqref{M3}.
 
Finally, let the system admit a sequence of weak martingale solutions 
\[
\mathcal{U}_n = 
\Big( (\Omega_n, \mathcal{A}_n, (\mathcal{A}_{n,t})_{t \geq 0}, 
\prst_n ) , \vr_n, \vu_n, \vc{B}_n, W^{1}_n, W^{2}_n \Big), 
\] 	
in $(0,T) \times \Td$
in the sense specified in Definition \ref{PD1} with initial conditions such that
\begin{align*}
\varrho_{0,n} &\xrightarrow{d} \varrho_0 \quad \text{in } L^1(\mathbb{T}^3), \\
\mathbf{m}_{0,n} &\xrightarrow{d} \mathbf{m}_0 \quad \text{in } L^1(\mathbb{T}^3; \mathbb{R}^3), \\
\mathbf{B}_{0,n} &\xrightarrow{d} \mathbf{B}_0 \quad \text{in } L^1(\mathbb{T}^3; \mathbb{R}^3), \\
\expe{\int_{\mathbb{T}^3} \left( E(\vr_{0,n}, \vm_{0,n}) +\frac{1}{2}|\vc{B}_{0,n}|^2 \right) \, \mathrm{d}x}
&\leq
\expe{\int_{Q_F} E(\vr_0, \vm_0)  \dx + 
\frac{1}{2}\intTd{ |\vc{B}_0|^2  }}
\end{align*}
as $n\to \infty$, where $\xrightarrow{d}$ means convergence in distribution and
\begin{equation*}
    \vr_0=0, \ \vm_0 =0 \ \text{for any } x \in \Td\setminus Q_F.
\end{equation*}

Then there exists a subsequence (not relabeled here for the sake of simplicity) of random variables 
\[
(\tilde{\vr}_n, \tilde{\vu}_n, \tilde{\vc{B}}_n, \tilde{W}^{1}_n, \tilde{W}^{2}_n )_{n =1}^\infty 
\] 
defined on a joint probability basis  
\[
(\Omega, \mathcal{A}, (\mathcal{A}_{t})_{t \geq 0}, 
\prst)   
\]	
satisfying: 
\begin{itemize}
	\item
\[
(\tilde{\vr}_n, \tilde{\vu}_n, \tilde{\vc{B}}_n, \tilde{W}^{1}_n, \tilde{W}^{2}_n ) \sim 
(\vr_n, \vu_n, \vc{B}_n, W^{1}_n, W^{2}_n) 
\]	
in 
\[
L^1((0,T) \times \Td; \R^7) \times C([0,T]; \R^2); 
\]
\item 
\begin{align} 
\tilde{\vr}_n &\to \vr \ \mbox{weakly-(*) in}\ L^\infty(0,T; L^\gamma (\Td) ), \br
\tilde{\vu}_n &\to \vu \ \mbox{weakly in}\ L^2(0,T; W^{1,2}(\Td; R^3)),	\br
\tilde{\vc{B}}_n &\to \vc{B} \ \mbox{weakly in} \
L^2(0,T; W^{1,2}(\Td; R^3)), \br 
\tilde{W}^1_n &\to W^1,\ \tilde{W}^2_n \to W^2 \ \mbox{in}\ C[0,T]  
\nonumber
\end{align}	
$\prst$-a.s., where 
\[
\mathcal{U} = \Big( (\Omega, \mathcal{A}, (\mathcal{A}_t)_{t \geq 0}, 
\prst ) , \vr, \vu, \vc{B}, W^1, W^2 \Big)
\]
is a weak martingale solution of the problem in the sense of Definition \ref{wD1}.

\end{itemize}	

\end{theorem}

Martingale weak solutions for the periodic problem were constructed by Wang \cite{WangH21}. We recall once more that there are two major gaps in Wang's paper: 

\begin{enumerate}
\item Wang considers the homogeneous boundary conditions for 
$\vc{B}$, namely
\[
\vc{B} = 0 \ \text{ on the boundary}.
\]
The problem, however, is ill--posed with these boundary conditions (to be more precise, the problem is overdetermined; see \cite{Lassner}).

\item Wang considers general nonlinear diffusion coefficients in the multiplicative noise. In this case, however, the correct weak formulation must be based on solenoidal test functions. This also produces a ``pressure'' term in Ohm's law, which is completely ignored in \cite{WangH21}. 
\end{enumerate}
	
Wang's result can therefore be expected to hold under the following conditions: 
(i) The boundary conditions are space periodic, (ii) the diffusion coefficients in the noise are compatible with the solenoidal constraint. 
This is exactly the case we considered.	

As Wang's result was inspired by \cite{BrFeHobook}, we shall suppose that the martingale solutions exist if the fourth moment of the initial data is bounded. Therefore, we assume \eqref{P1} to hold. We obtain the following results as a corollary.

\begin{corollary} \label{MC1}
Under the hypothesis of Theorem \ref{MT1}, the problems 
\eqref{eq:MHD} and \eqref{energ_nerovnost} 
admit a weak martingale solution in the sense of Definition \ref{wD1}.
\end{corollary}
\begin{remark}
    Since the fluid equations contain no stochastic integral, one may fix $\omega \in \Omega$ (similarly as in limit passage for fluid equations in Section \ref{proof}) and apply deterministic arguments. This allows the results of Theorem \ref{MT1} to be extended to Lipschitz domains $Q$ and $Q_F$, following the approach of \cite{2002DomainDependence}.
\end{remark}

\section{Proof of Theorem \ref{MT1}}
\label{proof}
The rest of the paper focuses on the proof of Theorem \ref{MT1}. The proof proceeds as follows: First, we derive uniform bounds in appropriate spaces using the energy inequality \eqref{P6}. Next, we shall show that the sequence of martingale solutions converges in the probabilistic sense, using the stochastic compactness method. After that, we pass to the limit in the weak formulations of both the induction equation and the fluid equations. Finally, we recover the energy inequality \eqref{w6} from Definition \ref{wD1} by passing to the limit.
\subsection{Uniform bounds}
We extract uniform bounds from the energy inequality \eqref{P6}. Let us take a sequence of weak martingale solutions according to Definition \ref{PD1} satisfying the assumptions of Theorem \ref{MT1}.

Raising the energy inequality to the power $r\geq4$ and making a simple adjustment, we obtain
\begin{align}
		&\left[  \int_{\Td} E(\vrn, \vrn \vun) (\tau, \cdot) \dx + 
		\intTd{ \frac{1}{2} |\vc{B}_n|^2 (\tau, \cdot) }  \right]^r  \br &\quad + 
		\left[ \int_0^\tau \intTd{ \big[ \mathbb{S}(\Grad \vun) : \Grad \vun +{ \eta}_n |\Curl \vc{B}_n |^2+ \beta_n |\vun|^2\big]} \dt \right]^r  \br 
		&\quad \aleq \left[  \int_0^\tau \intTd{ |\vc{B}_n|^2} \dt 
		+ \int_0^\tau \intTd{  |\Curl \vc{j}|^2 } \dt \right]^r +\left[\int_0^\tau \intTd{\vr \vu \cdot \Grad G}\dt \right]^r\br &\quad + \left[ \int_0^t \left( \intTd{ |\vc{B}_n|^2(s, \cdot) } \right) \D W^1_{n} \right]^r  + \left[ \int_0^t \left( \intTd{ \Curl \vc{j} \cdot 
			\vc{B}_n(s, \cdot) } \right) \D W^2_{n} \right]^r \br &\quad + \left[ \int_{\Td} E(\vr_n, \vrn \vu_n) (0, \cdot) \dx + 
		\intTd{ \frac{1}{2} |\vc{B}_n|^2 (0, \cdot) } \right]^r .
		\nonumber
\end{align}
Recall that $\aleq1$ denotes $\leq C$ for some constant $C$ depending on the data but not on $n$. 
After applying the supremum over the time interval $[0,\tau]$ for some $\tau\leq T$ and  then the expectation, we get
\begin{align}
		&\mathbb{E}\left[  \sup_{t \in [0,\tau]}\int_{\Td} E(\vrn, \vrn \vun) (\tau, \cdot) \dx + 
		\sup_{t \in [0,\tau]}\intTd{ \frac{1}{2} |\vc{B}_n|^2 (\tau, \cdot) }  \right]^r  \br &\quad + 
		\mathbb{E}\left[ \int_0^\tau \intTd{ \big[ \mathbb{S}(\Grad \vun) : \Grad \vun +{ \eta}_n |\Curl \vc{B}_n |^2+ \beta_n |\vun|^2\big]} \dt \right]^r  \br 
		&\quad \aleq \mathbb{E}\left[  \int_0^\tau \intTd{ |\vc{B}_n|^2} \dt 
		+ \int_0^\tau \intTd{  |\Curl \vc{j}|^2 } \dt \right]^r \br &\quad+\mathbb{E}\sup_{t \in [0,\tau]}\left[\int_0^t \intTd{\vr \vu \cdot \Grad G} \dt\right]^r +\mathbb{E} \sup_{t \in [0,\tau]}\left[ \int_0^t \left( \intTd{ |\vc{B}_n|^2(s, \cdot) } \right) \D W^1_{n} \right]^r \br &\quad+ \mathbb{E}\sup_{t \in [0,\tau]}\left[ \int_0^t \left( \intTd{ \Curl \vc{j} \cdot 
			\vc{B}_n(s, \cdot) } \right) \D W^2_{n} \right]^r \br &\quad + \mathbb{E}\left[ \int_{\Td} E(\vr_n, \vrn \vu_n) (0, \cdot) \dx + 
		\intTd{ \frac{1}{2} |\vc{B}_n|^2 (0, \cdot) } \right]^r .
		\nonumber
\end{align}

%%% (See that the very last term is bounded by our assumptions from Theorem \ref{MT1}, as we assume that all the initial conditions have the same law as in \eqref{P1})

It remains to bound the right-hand side. Let us start with the gravitational term. With the help of Young's inequality and the Cauchy–Schwarz inequality, we obtain
\begin{equation*}
    \sup_{t \in [0,\tau]}\int_0^t \intTd{\vrn \vun \cdot \Grad G} \dt \aleq \int_0^\tau \intTd{ | \vrn \vun|} \dt
\end{equation*}
\begin{equation*}
    \aleq \| \vrn \|_{L^1(0,T;L^1(\Td))}+\int_0^\tau\| \sqrt{\vrn} |\vun|\|^2_{L^2(\Td)}.
\end{equation*}
The $L^1$ norm of density is bounded, since for all $t\in [0,T]$ $\prst$-a.s.
\begin{equation*}
    \| \vrn(t)\|_{L^1(\Td)}=\| \vr_{0,n}\|_{L^1(\Td)}\leq C(\omega), \ \omega\in \Omega.
\end{equation*}
Thus, applying Fubini's theorem, we obtain
\begin{equation*}
    \mathbb{E}\left[\int_0^\tau \intTd{\vrn \vun \cdot \Grad G} \dt\right]^r\aleq 1 + \int_0^\tau \mathbb{E}\| \sqrt{\vrn} |\vun|\|^{2r}_{L^2(\Td)}.
\end{equation*}
For both stochastic integral terms, let us use the Burkholder–Davis–Gundy inequality, which gives
\begin{align*}
    &\mathbb{E} \sup_{t \in [0,\tau]}\left[ \int_0^t \left( \intTd{ |\vc{B}_n|^2(s, \cdot) } \right) \D W^1_{n} \right]^r \quad\br &+ \mathbb{E}\sup_{t \in [0,\tau]}\left[ \int_0^t \left( \intTd{ \Curl \vc{j} \cdot 
			\vc{B}_n(s, \cdot) } \right) \D W^2_{n} \right]^r\br
    &\quad \aleq \expe{ \left( \int_0^\tau \| \vc{B}_n \|_{L^2(\Td; R^3)}^4 \dt \right)^{\frac{r}{2}} }  + \expe{ \left( \int_0^\tau \| \Curl \vc{j} \|_{L^2(\Td; R^3)}^4 \dt \right)^{\frac{r}{2}} },
\end{align*}
where in the second term, we used Young's inequality and the Cauchy–Schwarz inequality. Finally, by applying Jensen's inequality
\begin{align*}
    &\mathbb{E} \sup_{t \in [0,\tau]}\left[ \int_0^t \left( \intTd{ |\vc{B}_n|^2(s, \cdot) } \right) \D W^1_{n} \right]^r \quad\br &+ \mathbb{E}\sup_{t \in [0,\tau]}\left[ \int_0^t \left( \intTd{ \Curl \vc{j} \cdot 
			\vc{B}_n(s, \cdot) } \right) \D W^2_{n} \right]^r\br
    &\quad \leq C(T,r) \left( \int_0^\tau \mathbb{E}\| \vc{B} \|_{L^2(\Td; R^3)}^{2r} \dt+ \int_0^\tau \mathbb{E}\| \Curl \vc{j} \|_{L^2(\Td; R^3)}^{2r} \dt\right) \br
   & \quad \aleq 1+ \int_0^\tau \mathbb{E}\| \vc{B} \|_{L^2(\Td; R^3)}^{2r}.
\end{align*}
This gives
\begin{align}
		&\mathbb{E}\left[  \sup_{t \in [0,\tau]}\int_{\Td} E(\vrn, \vrn \vun) (\tau, \cdot) \dx + 
		\sup_{t \in [0,\tau]}\intTd{ \frac{1}{2} |\vc{B}_n|^2 (\tau, \cdot) }  \right]^r  \br &\quad + 
		\mathbb{E}\left[ \int_0^\tau \intTd{ \big[ \mathbb{S}(\Grad \vun) : \Grad \vun +{ \eta}_n |\Curl \vc{B}_n |^2+ \beta_n |\vun|^2\big]} \dt \right]^r  \br 
		&\quad \aleq1+ \int_0^\tau\mathbb{E}\| \vc{B}_n \|_{L^2(\Td; \R^3)}^{2r}\dt+ \int_0^\tau \mathbb{E}\| \sqrt{\vrn} |\vun|\|^{2r}_{L^2(\Td)}\dt.
        \label{almostthere}
\end{align}	

It remains to estimate the two terms on the right-hand side. In particular, by Grönwall's lemma, the inequality \eqref{almostthere} implies
\begin{align}
    \expe{ \sup_{t \in [0,\tau]} \| \vc{B}_n \|^{2r}_{L^2(\Td; \R^3)} }& \aleq 
	1,\br
    \mathbb{E}\left[ \sup_{t \in [0,\tau]}\| \sqrt{\vrn} |\vun|\|^{2r}_{L^2(\Td)}\right]&\aleq 1.
    \label{u2}
\end{align}
Moreover, using \eqref{u2} in \eqref{almostthere} gives
\begin{align} 
\expe{ \sup_{t \in [0,T]} \left( \intTd{ E(\vr_n, \vrn \vun)(\tau, \cdot) } \right)^r } &\aleq 1, \br
\expe{ \left( \int_0^T \intTd{ \Big[ \mathbb{S}(\Grad \vun) : \Grad \vun 
	+  |\Curl \vc{B}_n |^2  
	+ \beta_n |\vun|^2 \Big] } \dt \right)^r } &\aleq 1. 
	\label{u3}
\end{align}
uniformly in $n$. From the first part of \eqref{u3}, we obtain
\begin{equation*}
    \mathbb{E}\left[ \sup_{t \in [0,T]}\| \sqrt{\vrn} |\vun|\|^{2r}_{L^2(\Td)} +   \sup_{t \in [0,T]} \| \vrn \|^{\gamma r}_{L^\gamma(\Td)}+ \sup_{\tau \in [0,T]} \| \vc{B}_n \|^{2r}_{L^2(\Td; \R^3)}\right]\aleq 1.
\end{equation*}

From the second part of \eqref{u3}, we obtain
\begin{align*}
    \mathbb{E} \| \Grad \vun\|^{2r}_{L^2((0,\tau)\times \Td)}&\aleq
    1
\end{align*}
as well as for the magnetic term, we use the divergence-free condition on $\vc{B}_n$. Thus, we get
\begin{equation*}
    \mathbb{E}\| \Grad \vc{B}_n\|^{2r}_{L^2((0,\tau)\times \Td)}=\expe{\left( \int_0^T \intTd{ \Big[ |\Curl \vc{B}_n |^2  
	\Big] } \dt \right)^r} \aleq 1.
\end{equation*}
Observe that the norms of velocity $\vun$ and the magnetic field $\vc{B}_n$ are also under control in space $L^2(0,T;W^{1,2}(\Td;\R^3))$ as
\begin{equation*}
    \mathbb{E}\|\vc{B}_n \|_{L^2(0,T;W^{1,2}(\Td;\R^3))}^{2r}\leq \expe{\sup_{t \in [0,T]} \| \vc{B}_n\|^{2r}_{L^2(\Td;\R^3)}\left( \int_0^T\dt\right)^{2r}}\leq C 
\end{equation*}
To control the full norm $W^{1,2}(\Td;\R^3)$ of $\vun$ and not just the gradient, note that we have, by the generalized Poincaré inequality (see \cite{FeireislEduard2017SLiT}, Proposition 2.2.)
\begin{equation*}
    \| \vun \|^2_{W^{1,2}(\Td;\R^3)} \aleq \|  \Grad \vun \|^2_{L^{2}(\Td;\R^{3\times 3})}+ \int_{\Td \setminus Q_F} | \vun |^2 \dx .
\end{equation*}

We summarize the estimates above in the following lemma.
\begin{lemma}[]\label{odhady}
    Let $(\vrn, \vun, \vc{B}_n)$ and $r\geq4$ be from Theorem \ref{MT1}. Then there exist $C\in \R$ independent of $n$, such that
\begin{align*}
\mathbb{E}\left( \sup_{t\in[0, T]} \lnorm\sqrt{\vr_{n}} |\vu_{n}|\rnorm_{L^2(\mathbb{T}^3)}^{2r}   \right)
&\le C, 
\qquad &
\mathbb{E}\left( n^r\lnorm\vu_n\rnorm^{2r}_{L^2(0,T;L^{2}(\mathbb{T}^3\setminus Q_F;\R^3))}   \right) 
&\le C, \\
 \mathbb{E}\left( \sup_{t\in[0, T]}\lnorm \vr_{n}\rnorm_{L^\gamma(\mathbb{T}^3)}^{\gamma r}   \right) 
&\le C, 
\qquad &
\mathbb{E}\left( \lnorm\vu_n\rnorm^{2r}_{L^2(0,T;W^{1,2}(\mathbb{T}^3;\R^3))}   \right) 
&\le C, \\
\mathbb{E}\left(   \sup_{t\in[0, T]}\lnorm\vc{B}_{n}\rnorm^{2r}_{L^2(\mathbb{T}^3;\R^3)} \right)
&\le C, 
\qquad &
\mathbb{E}\left(  \lnorm\vc{B}_n\rnorm^{2r}_{L^2(0,T;W^{1,2}(\mathbb{T}^3;\R^3))}  \right) 
&\le C.
\end{align*}
\end{lemma}
\subsection{Stochastic compactness}
\label{04-stoch-comp}
To apply the stochastic compactness method, we may choose either the more restrictive Skorokhod representation theorem (see \cite{billing}) or its generalization by Jakubowski (see \cite{jakubowski}). In both cases, we may pass to a new probability basis
$(\Omega, \mathcal{A}, (\mathcal{A}_t)_{t \geq 0}, \prst)$ and to a new sequence of
random variables having the same law as the original and converging a.s. 
(even surely in the Jakubowski version) in the appropriate spaces. Let us present the precise formulation of the Jakubowski representation theorem that will be used (see Theorem 2.7.1; \cite{BrFeHobook}).
\begin{theorem}[Jakubowski--Skorokhod representation theorem]\label{jakubowski}
Let $(X, \tau)$ be a sub-Polish space, and let $\mathcal{S}$ be the $\sigma$-field generated by $\{f_n; n \in \mathbb{N} \}$. If $(\mu_n)_{n \in \mathbb{N}}$ is a tight sequence of probability measures on $(X, \mathcal{S})$, then there exists a subsequence $(n_k)$ and $X$-valued Borel measurable random variables $(U_k)_{k \in \mathbb{N}}$ and $U$ defined on the standard probability space $([0,1], \mathcal{B}([0,1]), \mathfrak{L})$, such that $\mu_{n_k}$ is the law of $U_k$ and $U_k(\omega)$ converges to $U(\omega)$ in $X$ for every $\omega \in [0,1]$. Moreover, the law of $U$ is a Radon measure.
\end{theorem}

We apply Theorem \ref{jakubowski} to the joint law of the following random variables:
\begin{align}
\mathcal{V}_n = \Big( &\vrn(0, \cdot), (\vrn \vun)(0, \cdot)=\vm_n(0, \cdot), 
 E(\vrn, \vm_n )(0, \cdot), \vc{B}_n(0, \cdot), \vc{j}_n \br 
&\left.
\vr_n, \vu_n, \vrn \vun=\vm_n, p(\vr_n), \vc{B}_n, W^{1}_n, W_{n}^2, 
\left\| \sqrt{\beta_n} \vun \right\|_{L^2((0,T) \times \Td; \R^3)} \right),
\nonumber
\end{align}
considered in the space
\[
\mathcal{X} = \mathcal{X}_{\rm data} \times \mathcal{X}_{\vr} \times \mathcal{X}_\vu \times 
\mathcal{X}_{\vm} \times \mathcal{X}_p \times \mathcal{X}_\vc{B} \times [\mathcal{X}_W]^2 \times \mathcal{X}_\beta,  
\]
where
\[
\mathcal{X}_{\rm data} = L^1(\Td; \R^5) \times L^2(\Td, \R^3) \times  L^\infty(0,T; W^{1,2}(Q_{\text{int}};\R^3)) , 
\]
\[
\mathcal{X}_\vr = 
\left[ L^\infty(0,T; L^\gamma(\Td)) - \mbox{weak-(*)} \right], 
\]
\[
\mathcal{X}_\vu = \Big[ L^2(0,T; W^{1,2}(\Td; \R^3)) - \mbox{weak} \Big],
\]
\[
\mathcal{X}_{\vm} =   
\left[ L^\infty(0,T; L^{\frac{2 \gamma}{\gamma + 1}}(\Td;  \R^3)) - \mbox{weak-(*)} \right],
\]
\[
\mathcal{X}_p = \left[ L^\infty(0,T; \mathcal{M}(\Td)) - \mbox{weak-(*)} \right],
\]
\[
\mathcal{X}_\vc{B} = 
\Big[ L^\infty (0,T; L^2(\Td; \R^3)) - \mbox{weak-(*)} \Big] \cap \Big[ L^2(0,T; W^{1,2}(\Td; \R^3)) - \mbox{weak} \Big]
\]
\[
 \cap \,
C([0,T]; W^{-1,2}(\Td; \R^3)) ;
\]
\[
\mathcal{X}_W = C[0,T];\ 
\mathcal{X}_\beta = \R. 
\]

For notational convenience, in the compactness argument we denote by \(j_n\) a random process that has the same law as \(j\). Thus, \(j_n\) is not an approximation of \(j\); rather,
\(j_n\) has the same law as \(j\) for every \(n\), and satisfies the uniform bound \eqref{f5}.

It follows from the uniform bounds obtained from Lemma \ref{odhady} that the
joint distribution of the random variables in $\mathcal{V}_n$ is tight in $\mathcal{X}$.
Applying Theorem \ref{jakubowski}, we obtain a subsequence of random variables
\begin{align}
	\widetilde{\mathcal{V}}_n = \Big( &\trn(0, \cdot), \widetilde{\vm}_n(0, \cdot), 
	 E(\trn, \widetilde{\vm}_n )(0, \cdot), \tBn(0, \cdot), \widetilde{\vc{j}}_n,  \br 
	&\left.
	\trn, \tun, \widetilde{\vm}_n, p(\trn), \tBn,  \widetilde{W}^1_n, \widetilde{W}^2_n, 
	\left\| \sqrt{\beta_n} \tun \right\|_{L^2((0,T) \times \Td; \R^3)} \right)
	\nonumber
\end{align}
defined on the same probability basis 
\[
(\Omega, \mathcal{A}, \prst) \ \mbox{with a right--continuous 
	filtration} \ (\mathcal{A}_{n,t})_{t \geq 0}
\]
and 
\[
(\widetilde{W}^1_n, \widetilde{W}^2_n) -  (\mathcal{A}_{n,t}) \ \mbox{-Wiener processes},
\]
having the same law as $\mathcal{V}_n$ in $\mathcal{X}$. Moreover, it holds $\prst$-a.s. that
\begin{align}
\trn(0, \cdot) &\to \vr_0 \ \mbox{in}\ L^1(\Td), \br
\widetilde{\vm}_n (0, \cdot) &\to \vm_0 \ \mbox{in}\ L^1(\Td; R^3), \br
E(\trn, \widetilde{\vm}_n)(0, \cdot) &\to E = 
E(\vr_0, \vm_0) \ \mbox{in}\ L^1(\Td),\br
\widetilde{\vc{j}}_n &\to \vc{j} \ \mbox{weakly-(*) in}\ L^\infty(0,T; W^{1,2}_0 (Q_{\rm int}; \R^3)),\br
\tBn (0, \cdot) & \to \vc{B}_0 \ \mbox{in}\ L^2(\Td; R^3), \br
\trn &\to \vr \ \mbox{weakly-(*) in}\ L^\infty(0,T; L^\gamma(\Td)), \br
\tun &\to \vu \ \mbox{weakly in}\ L^2(0,T; W^{1,2}(\Td; R^3)), \br
\widetilde{\vm}_n &\to \vm \ \mbox{weakly-(*) in}\ L^\infty(0,T; L^{\frac{2\gamma}{\gamma +1}}(\Td;R^3)), \br
p(\trn) &\to \Ov{p(\vr)} \ \mbox{weakly-(*) in}\ L^\infty(0,T; \mathcal{M}(\Td)), \br
\tBn &\to \vc{B} \ \mbox{weakly-(*) in}\ L^\infty(0,T; L^2(\Td; \R^3)), \br &\quad \quad \ \ \mbox{weakly in}\ L^2(0,T; W^{1,2}(\Td; \R^3),\br
& \quad \quad \ \ \mbox{and in}\ C([0,T]; W^{-1,2}(\Td; \R^3)), \br
\tWni &\to W^1,\ \tWnii \to W^2 \ \mbox{in}\ C[0,T],\br 
\left\| \sqrt{\beta_n} \tun \right\|_{L^2((0,T) \times \Td; \R^3)} &\to M < \infty.
\label{SC1}
\end{align}
It is easy to show that $W^1$ and $W^2$ are also Wiener processes, and furthermore, that they are $(\mathcal{A}_{t}) \ \mbox{-Wiener processes}$ by the same method as in \cite{BrFeHobook} (Proposition 4.3.10). The last convergence in \eqref{SC1} means, in particular, that the sequence
\[
\left( \left\| \sqrt{\beta_n} \tun \right\|_{L^2((0,T)\times\Td; \R^3)} \right)_{n=1}^\infty \ \mbox{is bounded}. 
\]
Moreover, \eqref{SC1} yields the following pointwise bounds, as stated in Lemma \ref{odhady bez expe}.
\begin{lemma}[]\label{odhady bez expe}
\begin{align*}
\| \sqrt{\trn} |\tun|\|_{L^{\infty}(0,T;L^2(\mathbb{T}^3))}
&\le C 
\qquad &
\sqrt{n} \| \tun \|_{L^2(0,T;L^2(\mathbb{T}^3\setminus Q_F;\R^3)} 
&\le C, \\
\| \trn \|_{L^{\infty}(0,T;L^\gamma(\mathbb{T}^3))} 
&\le C 
\qquad &
\| \tun \|_{L^2(0,T;W^{1,2}(\mathbb{T}^3;\R^3))} 
&\le C, \\
\| \tBn\|_{L^{\infty}(0,T;L^2(\mathbb{T}^3;\R^3))}
&\le C 
\qquad &
\| \tBn \|_{L^{2}(0,T;W^{1,2}(\mathbb{T}^3;\R^3))} 
&\le C.
\end{align*}

where the constant $C\geq0$ may differ based on the choice of random event $\omega \in \Omega$.    
\end{lemma}
Thus, the only thing that is left to prove before passing the limit is the proof of tightness.
\\

To begin with, the data $\left(\vrn(0, \cdot), \vm_n(0, \cdot), 
\vc{B}_n(0, \cdot), E(\vrn, \vm_n )(0, \cdot), \vc{j}_n\right)$ are tight in $\mathcal{X}_{\text{data}}$ by Prokhorov's theorem (\cite{billing}, Theorem 5.1.). The tightness of the Wiener process $W$ in $\mathcal{X}_W=C[0,T]$ is a well-known result.

We next address the tightness of the momentum $\vrn \vun= \vm_n$ in the space
\begin{equation*}
    \mathcal{X}_{\vm}=\Big[ L^\infty(0,T; L^{\frac{2\gamma}{\gamma+1}}(\Td; \R^3)) - \mbox{weak-(*)} \Big],
\end{equation*}
with the norm defined as
\begin{equation*}
    \|f \|_{\mathcal{X}_{\vm}}= \sup_{t\in[0,T]}\|f \|_{L^{\frac{2\gamma}{\gamma+1}}(\Td; \R^3))}.
\end{equation*}
For our compact set $K_{\vm,\delta}\subset\mathcal{X}_{\vm}$, we shall choose the ball centered at the origin with radius $L_\delta$, i.e.
\begin{equation*}
    K_{\vm,\delta}:=\big\{f\in\mathcal{X}_{\vm}:  \|f \|_{\mathcal{X}_{\vm}}\leq L_\delta \big\},
\end{equation*}
where the constant $L_\delta$ shall be chosen later. The compactness of this domain holds by the Banach–Alaoglu theorem. Using Markov's inequality, we get
\begin{equation*}
    \mathbb{P}(\vm_n \notin K_{\vm,\delta})=\mathbb{P}(\|\vm_n \|_{\mathcal{X}_{\vm}}\geq L_\delta)
    \leq \frac{\mathbb{E}\|\vm_n \|_{\mathcal{X}_{\vm}}}{L_\delta}.
\end{equation*}
Thanks to Lemma \ref{odhady}, the term on the right-hand side is under control. Indeed, using Hölder's inequality two times, we even see that

\begin{align*}
    &\expe{\sup_{t\in[0, T]} \|\vm_n \|_{L^{q}(\Td;\R^3)}^{\frac{2\gamma}{\gamma+1}r}}\leq \expe{\sup_{t\in[0, T]}\| \vrn\|^{\frac{\gamma r}{\gamma+1}}_{L^\gamma(\Td)}\cdot\sup_{t\in[0, T]}\|\sqrt{\vrn}|\vun| \|^{\frac{2\gamma}{\gamma+1}r}_{L^2(\Td)}}\br
    &\quad\leq \left(\expe{\sup_{t\in[0, T]}\| \vrn\|^{\gamma r}_{L^\gamma(\Td)}} \right)^{\frac{1}{\gamma+1}}\left( \expe{\sup_{t\in[0, T]}\|\sqrt{\vrn}|\vun| \|^{2r}_{L^2(\Td)}}\right)^{\frac{\gamma}{\gamma+1}}\br
    &\quad \leq C(\text{DATA}),
\end{align*}
for
\begin{equation*}
    q=\frac{2\gamma}{\gamma+1}.
\end{equation*}
This is given by Jensen inequality (here we point out that $qr >1$ for all $\gamma> \frac{3}{2}$)
\begin{equation*}
    \mathbb{P}(\vm_n\notin K_{\vm,\delta})\leq \frac{C}{L_\delta}.
\end{equation*}
In the same way as before, choosing $L_\delta=\frac{9C}{\delta}$ yields the result
\begin{equation}
    \mathbb{P}(\vrn \vun \notin K_{\vm,\delta})=\mathbb{P}(\vm_n \notin K_{\vm,\delta})\leq\frac{\delta}{9}.
\end{equation}
We can proceed similarly by choosing the ball as a compact set but more straightforwardly with the following spaces: $\mathcal{X}_\beta:=\R$, $\mathcal{X}_{\vr}:= \left[ L^\infty(0,T; L^{\gamma}(\Td)) - \mbox{weak-(*)} \right]$, $\mathcal{X}_\vu:=\left[ L^2(0,T; W^{1,2}(\Td; \R^3)) - \mbox{weak} \right]$, and lastly $\mathcal{X}_p := \left[ L^\infty(0,T; \mathcal{M}(\Td)) - \mbox{weak-(*)} \right]$, where $\mathcal{M}(\Td)$ denotes the space of Radon measures on $\Td$.
\\

Let us point out two things. The space $\mathcal{X}_\vu$ is not Polish, as it is not completely metrizable; thus, the Jakubowski theorem is needed here rather than the classical Skorokhod theorem.

Also note that, in the case of $\mathcal{X}_p$, we use the fact
\begin{equation*}
    L^\infty(0,T; L^1(\Td)) \hookrightarrow L^\infty(0,T; \mathcal{M}(\Td))\subset (L^1(0,T;C(\Td))^*,
\end{equation*}
Thus, the closed ball in $\mathcal{X}_p$ is weakly-(*) compact by the Banach–Alaoglu theorem on a closed subset.
\\

Lastly, let us prove the tightness of the magnetic field $\vc{B}_n$ in the space $\mathcal{X}_\vc{B}$ defined as
\begin{equation*}
    \mathcal{X}_\vc{B} := 
C([0,T]; W^{-1,2}(\Td; \R^3)) \cap \Big[ L^\infty (0,T; L^2(\Td; \R^3)) - \mbox{weak-(*)} \Big]
\end{equation*}
\[
\cap \Big[ L^2(0,T; W^{1,2}(\Td; \R^3)) - \mbox{weak} \Big],
\]
with the norm
\begin{equation*}
    \| f \|_{\mathcal{X}_\vc{B}}= \sup_{t\in[0,T]}\| f\|_{L^2(\Td; \R^3))}+\| f\|_{L^2(0,T; W^{1,2}(\Td; \R^3))}+ \|f \|_{C([0,T]; W^{-1,2}(\Td; \R^3))}.
\end{equation*}
Let us choose our compact set $K_{\vc{B},\delta}\subset\mathcal{X}_{\vc{B}}$ as
\begin{equation*}
    K_{\vc{B},\delta}:=\big\{f\in\mathcal{X}_{\vc{B}}:  \|f \|_{\mathcal{X}_{\vc{B}}}\leq L_\delta; \text{ f is equicontinuous in } C([0,T];W^{-1,2}(\Td;\R^3))\big\}.
\end{equation*}
The equicontinuity of $\vc{B}_n$ will be established below.

Now we shall prove the tightness. Note that the norm of $\mathcal{X}_{\vc{B}}$ is mostly under control. The norms of $L^\infty(0,T; L^{2}(\Td))$ and $L^2(0,T; W^{1,2}(\Td))$ are controlled thanks to Lemma \ref{odhady}. It remains to establish tightness with respect to space $C([0,T];W^{-1,2}(\Td;\R^3))$. We have
\begin{align}
    \vc{B}_n(t,\cdot) =& \vc{B}_n(0,\cdot) - \int_0^t\left[ \Curl(\vun \times \vc{B}_n) + \Curl(\eta_n \Curl \vc{B}_n)\right] \D s \br
    +&\alpha\int_0^t \vc{B}_n \, \D W^1_n+\int_0^t \Curl \vc{j}_n\, \D W^2_n.
    \label{z_1}
\end{align}

First of all, we prove the time regularity of $\vc{B}_n$, which holds true uniformly in $n$. Let us start with the deterministic part of \eqref{z_1}, namely,
\begin{equation*}
    \vc{D}_n(t,\cdot)=\vc{B}_n(0,\cdot) - \int_0^t\left[ \Curl(\vun \times \vc{B}_n) + \Curl(\eta_n \Curl \vc{B}_n)\right] \D s.
\end{equation*}

We want to show
\begin{equation}\label{z_2}
    \mathbb{E}\| \vc{D}_n \|_{C([0,T];W^{-1,2}(\Td;\R^3))}\leq C.
\end{equation}

For this, we will use the following well-known embeddings:
\begin{equation}\label{compact-embed1,2}
    W^{1,2}(\Td)\hookrightarrow\hookrightarrow L^2(\Td)\hookrightarrow\hookrightarrow W^{-1,2}(\Td),
\end{equation}
\begin{equation}\label{embed - do holder spoj}
    W^{p,q}(0,T;H) \hookrightarrow C^z([0,T];H) \ \text{ whenever } \, z \in\left[0,p-\frac{1}{q} \right),
\end{equation}
for some Hilbert space $H$. For the first term, it holds simply by using embedding \eqref{compact-embed1,2}
\begin{equation}
    \| \vc{B}_{0,n}\|_{W^{-1,2}(\Td;\R^3)}\leq \|\vc{B}_{0,n} \|_{L^2(\Td;\R^3)}.
\end{equation}
For the second term of $\vc{D}_n$, we see that using Hölder's inequality
\begin{align*}
    \int_0^\tau \|\vun \times \vc{B}_n \|^2_{L^{\frac{3}{2}}(\Td;\R^3)} \dt &\leq \int_0^\tau \|\vun \|_{L^6(\Td;\R^3)}^2 \|\vc{B}_n \|_{L^2(\Td;\R^3)}^2 \dt \br
    & \leq \left( \sup_{t \in [0,T]}\|\vc{B}_n \|_{L^2(\Td;\R^3)}^2\right)\int_0^\tau \|\vun \|_{W^{1,2}(\Td;\R^3)}^2 \dt.
\end{align*}
After the possible application of expectation $\mathbb{E}$, we would get the right-hand side bounded; therefore,
\begin{equation*}
    \vun \times \vc{B}_n \ \text{ is bounded in } L^2(0,T;L^{\frac{3}{2}}(\Td;\R^3)) 
\end{equation*}
$\prst$-a.s., which gives
\begin{equation*}
    \Curl(\vun \times \vc{B}_n) \ \text{ is bounded in } L^2(0,T;W^{-1,\frac{3}{2}}(\Td;\R^3))
\end{equation*}
$\prst$-a.s.

For the last term, by \eqref{M3} we can achieve the following
\begin{equation*}
    \int_0^\tau \|\eta_n \, \Curl \vc{B}_n \|^2_{L^2(\Td;\R^3)} \dt \aleq \int_0^\tau \| \vc{B}_n \|^2_{W^{1,2}(\Td;\R^3)}.
\end{equation*}
Therefore, similar to the previous case, we have
\begin{equation*}
    \Curl(\eta_n \, \Curl \vc{B}_n) \ \text{ is bounded in } L^2(0,T;W^{-1,2}(\Td;\R^3))
\end{equation*}
$\prst$-a.s.

This yields that $\vc{D}_n$ is bounded $\prst$-a.s. in
\begin{equation*}
    W^{1,2}(0,T;W^{-1,2}(\Td;\R^3)),
\end{equation*}
due to the Rellich–Kondrachov theorem and the additional regularity in time from the integration over time. Using \eqref{embed - do holder spoj} for $z=0$, we get \eqref{z_2}. By the definition of $\vc{D}_n$ and the bounds above, it is clearly equicontinuous.

The remaining terms are stochastic integrals. Therefore, in both cases we use the generalization of the Burkholder-Davis-Gundy inequality (see \cite{BrFeHobook}, Proposition 2.3.8):
\begin{equation}\label{Gundy2}
    \mathbb{E} \sup_{t\in[0,T]} \left\lVert \int_0^t \vc{G}(s) \,\D W(s)\right\rVert_{H}^p
    \leq
    C_p \, \mathbb{E}\left( \int_0^T \|\vc{G}(s) \|^2_{H}\, \D s\right)^{\frac{p}{2}},
\end{equation}
where $H$ is a separable Hilbert space and $\|\cdot \|_{H}$ its norm, $p\in (0,\infty)$, $C_p>0$, and $\vc{G}$ are $(\mathcal{F}_t)$-progressively measurable stochastic processes such that
\begin{equation*}
    \mathbb{E}\int_0^T \|\vc{G}(t) \|^2_{H} \D t <\infty.
\end{equation*}

Now, it suffices to check the time regularity of the stochastic integral. Applying \eqref{Gundy2} and \eqref{f5}, we get
\begin{align*}
    \expe{\left\Vert \int_{\tau_1}^{\tau_2} \Curl \vc{j}_n \, \D W^2_n \right\Vert^q_{L^2(\Td;\R^3)}}&\aleq \expe{\int_{\tau_1}^{\tau_2} \|\Curl \vc{j}_n \|_{L^2(\Td;\R^3)}^2 \dt}^{\frac{q}{2}}\br
    &\aleq |\tau_2-\tau_1|^{\frac{q}{2}}\expe{\sup_{t\in [0,T]}\|\Curl \vc{j}_n \|_{L^2(\Td;\R^3)}^2}\br
    &\aleq K_1 |\tau_2-\tau_1|^{\frac{q}{2}}, \ \text{ for every } q\geq 1.
\end{align*}
Note that the last line gives us equicontinuity in space $C([0,T];L^2(\Td;\R^3))$, which provides the equicontinuity we desire by applying \eqref{compact-embed1,2}. Using the Kolmogorov continuity theorem and the embedding \eqref{compact-embed1,2}, we get
\begin{equation*}
    \mathbb{E}\left\Vert\int_0^\cdot \Curl \vc{j}_n \, \D W^2_n  \right\Vert_{C([0,T];W^{-1,2}(\Td;\R^3))}\leq C.
\end{equation*}

Proceeding in the same way, we obtain the second stochastic integral
\begin{align*}
    \expe{\left\Vert \int_{\tau_1}^{\tau_2} \alpha\, \vc{B}_n \, \D W^1_n\right\Vert^q_{L^2(\Td;\R^3)}}
    &\aleq |\tau_2-\tau_1|^{\frac{q}{2}}\expe{\sup_{t\in [0,T]}\|\vc{B}_n \|_{L^2(\Td;\R^3)}^2}\br
    &\aleq K_2 |\tau_2-\tau_1|^{\frac{q}{2}}, \ \text{ for every } q\geq 1.
\end{align*}
Using the Kolmogorov continuity theorem once more and the embedding \eqref{compact-embed1,2}, we get
\begin{equation*}
    \mathbb{E}\left\Vert\int_0^\cdot \alpha\,\vc{B}_n \, \D W^1_n  \right\Vert_{C([0,T];W^{-1,2}(\Td;\R^3))}\leq C.
\end{equation*}

Summarizing our results, this yields the desired regularity as well as
\begin{equation}\label{z_3}
    \mathbb{E}\| \vc{B}_n \|_{C([0,T];W^{-1,2}(\Td;\R^3))}\leq C,
\end{equation}
which implies
\begin{align*}
    \mathbb{P}(\vc{B}_n \notin K_{\vc{B},\delta})=\mathbb{P}(\|\vc{B}_n\|_{\mathcal{X}_{\vc{B} }}\geq L_\delta)
    &\leq \frac{\mathbb{E}\|\vc{B}_n \|_{\mathcal{X}_{\vc{B}}}}{L_\delta} \br
    &\leq \frac{C(\text{DATA})}{L_\delta},
\end{align*}

as the norm on the right-hand side has all its components controlled by data. Setting
\begin{equation*}
    L_\delta=\frac{9C}{\delta}
\end{equation*}
yields the tightness of $\vc{B}_n$ in space $\mathcal{X}_{\vc{B}}$.
\\

To conclude, let $\delta\in (0,1)$ be given. Then, from the previous calculations, there exist compact sets $K_{\text{data},\delta}\subset \mathcal{X}_{\text{data}}$, $K_{\vr,\delta}\subset \mathcal{X}_{\vr}$, $K_{\vu,\delta}\subset \mathcal{X}_{\vu}$, $K_{\vr \vu,\delta}\subset \mathcal{X}_{\vr \vu}$, $K_{p,\delta}\subset \mathcal{X}_{p}$, $K_{\vc{B},\delta}\subset \mathcal{X}_{\vc{B}}$, $K_{W,\delta}\subset \mathcal{X}_{W}$, and $K_{\beta,\delta}\subset \mathcal{X}_{\beta}$. Due to Tychonoff's theorem, the set 
\begin{equation*}
    K_\delta:=K_{\text{data},\delta}\times K_{\vr,\delta}\times K_{\vu,\delta}\times K_{\vr \vu,\delta}\times K_{p,\delta}\times K_{\vc{B},\delta}\times K_{W^1,\delta}\times K_{W^2,\delta}\times K_{\beta,\delta}
\end{equation*}
is compact in $\mathcal{X}$. Using De Morgan's law,
\begin{equation*}
    A\,\cap B=(A^C\cup B^C)^C,
\end{equation*}
we get the tightness. Indeed, we have
\begin{align*}
    \mathbb{P} \Biggl[ \bigg(d_n\times&\vrn \times \vun \times \vrn \vun \times p(\vrn) \times \vc{B}_n \times W^1_n \times W^2_n \times\left\| \sqrt{\beta_n} \vun \right\|_{L^2(\Td; \R^3)}\bigg) \in K_\delta\Biggl]\br
    \geq1& - \mathbb{P}(d_n \notin K_{\text{data}),\delta}- \mathbb{P}(\vrn \notin K_{\vu,\delta})-\mathbb{P}(\vun \notin K_{\vu,\delta})-\mathbb{P}(\vrn \vun \notin K_{\vr \vu,\delta}) \br
    &- \mathbb{P}(p(\vrn) \notin K_{p,\delta})-\mathbb{P}(\vc{B}_n\notin K_{\vc{B},\delta}) - \mathbb{P}(W^1_n\notin K_{W,\delta})-\mathbb{P}(W^2_n\notin K_{W,\delta}) \br
    & -\mathbb{P}\left(\left\| \sqrt{\beta_n} \vun \right\|_{L^2(\Td; \R^3)}\notin K_{\beta,\delta} \right)\br
    \geq 1& - \delta,
\end{align*}
where $d_n$ is defined as
\begin{equation*}
    d_n:=\big(\vrn(0, \cdot), \vm_n(0, \cdot), 
\vc{B}_n(0, \cdot), E(\vrn, \vm_n )(0, \cdot), \vc{j}_n\big).
\end{equation*}

\subsection{Limit passage}
From now on, we shall work with our new subsequence $\Tilde{\mathcal{V}}_n$ from \eqref{SC1}. For the sake of convenience, we shall drop the tilde mark and still denote our variables in the same way unless stated otherwise.

Let us first focus on the induction equation. Observe that the convergences in \eqref{SC1} yield the following results.
\begin{lemma}[]\label{konvergence-magnet.field}
    We have $\prst$-a.s.

\noindent
$\bullet$ \quad 
$\vu_n \to \vu \quad \text{weakly in } L^2(0,T;W^{1,2}(\Td;\R^3)).$

\vspace{0.5em}

\noindent
$\bullet$ \quad 
$\vu_n \to \vc{0} \quad \text{strongly in } L^2((0,T) \times \mathbb{T}^3 \setminus Q_F), \quad 
\vu \in L^2(0,T; W^{1,2}_0(Q_F; \mathbb{R}^3)).$

\vspace{0.5em}

\noindent
$\bullet$ \quad 
$\vc{B}_n \to \vc{B}
\quad \text{strongly in } L^2((0,T) \times \mathbb{T}^3; \mathbb{R}^3) $.

\vspace{0.5em}

\noindent
$\bullet$ \quad 
$\vc{B}_n \times \vu_n \to \vc{B} \times \vu 
\quad \text{weakly in } L^s((0,T) \times \mathbb{T}^3; \mathbb{R}^3) 
\quad \text{for some } s > 1.$

\vspace{0.5em}

\noindent
$\bullet$ \quad 
$\vc{B}_n \rightarrow \vc{B} \ \mbox{in}\ C_{\rm {weak}}([0,T]; L^2(\Td; \R^3)).$
\end{lemma}
Before giving the proof, note that by the last point of Lemma \ref{konvergence-magnet.field} we can $\prst$-a.s. identify
\begin{equation*}
    \vc{B}(0, \cdot) = \vc{B}_0.
\end{equation*}
\begin{proof}
    The first statement follows immediately from \eqref{SC1}. The second statement follows from \eqref{SC1} and the definition of $\beta_n$. Indeed, it holds
\begin{equation*}
    \| \vun\|^2_{L^2((0,T)\times\Td\setminus Q_F)}\leq  \frac{C(\omega)}{n},
\end{equation*}
    where the constant $C$ may depend on $\omega\in \Omega$. Nevertheless, by fixing $\omega$ and considering only its trajectory, we get the second statement by passing to the limit $n \to \infty$.

    For the strong convergence of $\vc{B}_n$, we shall use the weak convergence in $L^2(0,T;W^{1,2}(\Td;\R^3))$ and the convergence in $C([0,T];W^{-1,2}(\Td;\R^3))$. First, note that
    \begin{equation*}
        C[0,T] \subset L^2[0,T],
    \end{equation*}
    which implies
    \begin{align*}
        \lim_{n\to \infty} \int_0^\tau \intTd{\vc{B}_n \cdot \vc{B}_n} \dt
        &=\int_0^\tau \intTd{|\vc{B}|^2}\dt,
    \end{align*}
    where we used the duality of the spaces $W^{1,2}(\Td;\R^3)$ and $W^{-1,2}(\Td;\R^3)$.

    For the weak convergence of $\vc{B}_n \times \vun$, we shall first establish the bounds independent of $n$. By \eqref{SC1}, it can be shown
    \begin{equation*}
         \int_0^T \| \vc{B}_n \times \vun \|^2_{L^1(\Td;\R^3)} \dt + \int_0^T \| \vc{B}_n \times \vun \|_{L^3(\Td;\R^3)} \dt\leq C(\omega).
    \end{equation*}
By interpolation, we get the weak convergence in $L^s((0,T) \times \Td; \R^3))$ for some $s>1$ to some limit. To identify the weak limit, we shall use the convergence in $C([0,T];W^{-1,2}(\Td;\R^3))$ and the weak convergence for $\vun$ to identify the limit.

Lastly, to strengthen the convergence of $L^\infty(0,T;L^2(\Td;\R^3))$, we need to show
\begin{equation*}
        \int_{\Td} \vc{B}_n \cdot \bm{\varphi} \, \dx \to \int_{\Td} \vc{B} \cdot \bm{\varphi} \, \dx \ \text{ in } C([0,T]), \ \forall \bm{\varphi} \in C^\infty(\Td;\R^3),
\end{equation*}
but this holds true due to convergence in $C([0,T];W^{-1,2}(\Td;\R^3))$ as
\begin{equation*}
    C^\infty(\Td;\R^3) \subset (W^{-1,2}(\Td;\R^3))'=W^{1,2}(\Td;\R^3).
\end{equation*}
\end{proof}

Finally, it follows from the equality of laws that our new $\vc{B}_n$, $\vun$ solve the (stochastic) induction equation $\prst$-a.s.:  
\begin{align} 
    \left[ \intTd{ \vc{B}_n \cdot  \bm{\Psi} } \right]_{t = 0}^{t = \tau} &=\int_0^\tau\intTd{\vc{B}_n \cdot \partial_t \bm{\Psi} }\dt\br
 &-\int_0^\tau\int_{\Td}[  \eta_n \, \Curl \vc{B}_n\cdot \Curl\bm{\Psi}
 +(\vc{B}_n\times\vun )\cdot \Curl\bm{\Psi}]\D x \, \D t \br
 &+\int_0^\tau \left(\int_{\Td}\alpha\vc{B}_n \cdot \bm{\Psi}\, \D x\right) \D W^1_n
    \br 
    &+\int_0^\tau  \left(\int_{\Td}  \Curl\vc{j} \cdot\bm{\Psi} \, \D x\right)\D W^2_n
	\label{A1}	
\end{align}		
for any $\bm{\Psi} \in C^1([0,T] \times \Td;\R^3)$, where 
\[
\vc{B}_n(0, \cdot) \to \vc{B}_0 \ \mbox{in}\ L^2(\Td; \R^3),
\] 
and
\begin{equation} \label{A2}
	\intTd{ \vc{B}_n (\tau, \cdot) \cdot \Grad \psi } = 0 
\end{equation} 
for any $\psi \in C^1(\Td)$. 

We now pass to the limit $n\rightarrow\infty$. Note that we can pass to the limit in all the non-stochastic integral terms on the right-hand side of \eqref{A1} by applying the limits from Lemma \ref{konvergence-magnet.field} above.

For the stochastic integral terms, we apply Lemma 2.6.5 from \cite{BrFeHobook}.
\begin{lemma}[]\label{limit-stoch integ}
    Let $(\Omega, \mathfrak{F}, \mathbb{P})$ be a complete probability space. For $n\in \mathbb{N}$, let $W_n$ be an $(\mathfrak{F}^n_t)$-Wiener process, and let $G_n$ be a real $(\mathfrak{F}^n_t)$-progressively measurable stochastic process such that $G_n \in L^2(0,T)$ $\mathbb{P}$-a.s. Suppose that
    \begin{align*}
        W_n&\rightarrow W \ \text{ in } C([0,T]) \ \text{in probability},\br
        {G}_n &\rightarrow {G} \ \text{ in } L^2(0,T) \ \text{in probability},
    \end{align*}
    where $W$ is a Wiener process adapted to a filtration $(\mathfrak{F}_t)_{t\geq0}$, given by
    \begin{equation*}
        \mathfrak{F}_t=\sigma\left( \sigma_t[G] \cup \sigma_t[W]\right)
    \end{equation*}
 Then, after a possible change on a set of zero measure in $\Omega\times(0,T)$, $G$ is $(\mathfrak{F}_t)$-progressively measurable, and
    \begin{equation*}
        \int_0^\cdot G_n \, \D W_n \rightarrow \int_0^\cdot G \, \D W \ \text{ in } L^2(0,T) \ \text{in probability}.
    \end{equation*}
\end{lemma}

Observe that the convergence in probability in Lemma \ref{limit-stoch integ} is not a problem, as it is a well-known fact that we can always find a subsequence that converges almost surely.
\\

Choosing
\begin{equation*}
    G_n= \intTd{ \alpha \vc{B}_n \cdot \bm{\Psi} }
\end{equation*}
gives 
\begin{align*}
 \int_0^\tau \left( \intTd{ \alpha \vc{B}_n \cdot \bm{\Psi} } \right) \D W^1_n \rightarrow& \int_0^\tau \left( \intTd{ \alpha \vc{B} \cdot \bm{\Psi} } \right) \D W^1,
\end{align*}
as
\begin{align*}
    \mathbb{E}\int_0^T \left|\intTd{\alpha (\vc{B}_n-\vc{B}) \cdot \bm{\Psi}} \right|^2 \dt&\aleq \mathbb{E} \int_0^T \|\vc{B}_n-\vc{B} \|^2_{L^2(\Td;\R^3)} \dt \br
    &= \mathbb{E} \|\vc{B}_n-\vc{B} \|^2_{L^2((0,T)\times\Td;\R^3)},
\end{align*}
which goes to zero by Lemma \ref{konvergence-magnet.field}. By lemmas \ref{limit-stoch integ} and the convergence of the Wiener process from \eqref{SC1}, we get the convergence of the stochastic integral. The second stochastic integral follows analogously, with
\begin{equation*}
    G_n=\intTd{\Curl \vc{j}_n \cdot \bm{\Psi}}.
\end{equation*}

Consequently, we get \eqref{w4} and \eqref{w5} from Definition \ref{wD1}.
\\

Turning to the fluid equations (i.e., the equation of momentum and the equation of continuity), we fix $\omega\in \Omega$ and work with individual trajectories. Due to this fact, the proof is the same as in the deterministic case, as was done in \cite{FePokTryKar} or in \cite{FeNeSt} with penalization. Note that this approach is only possible because there is no stochastic integral in fluid equations.

Due to the equivalence in law, we have $\mathbb{P}$-a.s.
	\begin{equation} \label{F2} 
	\left[ \intTd{ \vrn\, \varphi } \right]_{t=0}^{t = \tau} 
	= \int_0^\tau \intTd{ \Big[ \vrn\, \partial_t \varphi + \vm_n \cdot 
		\Grad \varphi \Big] } \dt 
\end{equation}	
for any $\varphi \in C^1([0,T] \times \Td)$. Similarly, as in the previous section, we can strengthen our convergence for $\vrn$ acquired from Lemma \ref{odhady}. Namely, using \eqref{F2}, it holds $\mathbb{P}$-a.s.
\begin{equation}\label{C_W(L^gamma)}
        \vrn \to \vr \ \mbox{in}\ C_{\rm {weak}}([0,T]; L^\gamma(\Td)).
\end{equation}

By equivalence of laws, the momentum equation holds $\mathbb{P}$-a.s. 
\begin{align}
	\left[ \intTd{ \vm_n \cdot \bfphi } \right]_{t = 0}^{t = \tau} &= \int_0^\tau \intTd{ \Big[ \vm_n \cdot \partial_t \bfphi + 
		{\vm_n \otimes \vun} :\Grad \bfphi + {p(\vrn)} \Div \bfphi \Big] } \dt \br &+\int_0^\tau \intTd{ {\vrn \Grad G \cdot \bfphi }}- \int_0^\tau \intTd{ \mathbb{S}(\Grad \vun) : \Grad \bfphi } \dt  \br
        &+ 
	\int_0^\tau \intTd{ {(\Curl \vc{B}_n \times \vc{B}_n}) \cdot \bfphi } \dt- \int_0^\tau \intTd{\beta_n \, \vun\cdot \bfphi}\dt\br
	\label{z2}
\end{align}	
for any $\bfphi \in C^\infty(\Td;\R^3)$. Similarly to \eqref{C_W(L^gamma)}, we get $\mathbb{P}$-a.s.
\begin{equation}\label{C_W(L^q)}
        \vm_n \to \vm \ \mbox{in}\ C_{\rm {weak}}\left([0,T]; L^\frac{2\gamma}{\gamma+1}(Q_F;\R^3)\right).
    \end{equation}

Here we point out that we cannot achieve weak convergence \eqref{C_W(L^q)} for the whole space $\mathbb{T}^3$, but only for the domain $Q_F$, as the penalization term can be bounded at most by
    \begin{equation*}
        \left|n\int_{\mathbb{T}^3\setminus Q_F}\vu_n \cdot\bm{\varphi} \D x \right|\leq C\sqrt{n}, 
    \end{equation*}
    which diverges as $n\rightarrow\infty$.

Exactly as in the previous section, \eqref{C_W(L^gamma)}, \eqref{C_W(L^q)} help us identify our initial condition
\[
\vr(0, \cdot) = \vr_{0}, \ \vm(0, \cdot) = \vm_0,\ \prst\text{-a.s}.
\]

Now we are ready to pass to the limit in the continuity equation \eqref{F2}. Using Lemma \ref{odhady}, \eqref{C_W(L^gamma)}, \eqref{C_W(L^q)}, and the assumption on convergence of initial conditions. Sending  $n \rightarrow \infty$, we obtain
\begin{equation} \label{F3} 
	\left[ \intTd{ \vr \varphi } \right]_{t=0}^{t = \tau} 
	= \int_0^\tau \intTd{ \Big[ \vr \partial_t \varphi + \vm \cdot 
		\Grad \varphi \Big] } \dt 
\end{equation}	
for any $\varphi \in C^1([0,T] \times \Td)$ $\prst$-a.s. 

So far, there has been no need to extract another subsequence. Observe that for the limit \eqref{F3}, it holds
\begin{equation}\label{nezapornost hustoty}
    \| \vr (t,\cdot)\|_{L^1(\Td)} =\| \vr_0(\cdot)\|_{L^1(\Td)}
\end{equation}
by choice $\varphi = 1$ and using the assumption \eqref{P1}.
\\

In particular, when choosing the test functions $\varphi \in C^1([0,T] \times \Ov{Q}_F)$, we obtain
\begin{equation*}
    \left[ \int_{Q_F} \vr \varphi \ dx \right]_{t=0}^{t = \tau} 
= \int_0^\tau \int_{Q_F} \Big[ \vr \partial_t \varphi + \vr \vu \cdot 
\Grad \varphi \Big] \ dx \dt,
\end{equation*}
which is exactly the formulation of the equation of continuity \eqref{w2} as in Definition \ref{wD1}.

To be more precise, we need to show that the limiting density and velocity go to zero outside $Q_F$. Indeed, from Lemma \ref{konvergence-magnet.field}, we know $\prst$-a.s.
\begin{equation*}
    \vu_n \rightarrow \vc{0} \ \text{strongly in } L^2((0,T)\times\mathbb{T}^3\setminus Q_F;\R^3).
\end{equation*}
Thus, from the weak formulation of the continuity equation \eqref{F3}, we get $\prst$-a.s.
\begin{equation*}
    0= \int_0^T\int_{\mathbb{T}^3\setminus Q_F}\vr \partial_t \varphi\, \D x\, \D t,\ \forall \varphi \in C_c^1([0,T) \times (\mathbb{T}^3\setminus Q_F)),
\end{equation*}
which implies that 
\begin{equation*}
    \vr(t,\cdot)=\vr_0(\cdot) \ \text{ in } (0,T)\times(\mathbb{T}^3\setminus Q_F).
\end{equation*}
Thus, we may identify our weak limit of $\vr_n$ as zero. This implies the strong convergence in space $L^1((0,T)\times \Td \setminus Q_F)$ as formulated in the following trivial observation.
\begin{lemma}\label{strongL1}
    Let $Q\subset \R^3$ be a bounded Lipschitz domain, $T\in \R$ and $\{g_n\}_n$ be a family of non-negative functions such that $g_n\rightarrow 0$ in $C_{\text{weak}}([0,T];L^q(Q))$ for $q > 1$. Then we have
    \begin{equation*}
        g_n\rightarrow 0 \ \text{ strongly in } L^1((0,T)\times Q).
    \end{equation*}
\end{lemma}

Applying Lemma \ref{strongL1}, we will obtain
\begin{equation*}
    \vr_n\rightarrow 0 \ \text{ strongly in } L^1((0,T)\times\mathbb{T}^3\setminus Q_F).
\end{equation*}
By interpolation, it gives
\begin{equation}\label{silna konverg}
    \| \vr_n\|_{L^q((0,T)\times \mathbb{T}^3\setminus Q_F)}\leq \| \vr_n\|_{L^1((0,T)\times \mathbb{T}^3\setminus Q_F)} ^\alpha \| \vr_n\|_{L^\gamma((0,T)\times \mathbb{T}^3\setminus Q_F)}^{1-\alpha},
\end{equation}
\begin{equation*}
    \text{for }\ \frac{1}{q}=\alpha+\frac{1-\alpha}{\gamma},\ \text{ where } \alpha \in [0,1].
\end{equation*}
Consequently, it gives
\begin{equation}\label{z_5}
    \vr_n\rightarrow 0 \ \text{ strongly in } L^q((0,T)\times \mathbb{T}^3\setminus Q_F)\text{ for any } 1\leq q<\gamma.
\end{equation}
We point out that the strong convergence and \eqref{nezapornost hustoty} imply
\begin{equation}\label{zachovani hmoty}
    \| \vr (t,\cdot)\|_{L^1(Q_F)} =\| \vr_0(\cdot)\|_{L^1(Q_F)}.
\end{equation}

We now pass to the limit in the momentum equation \eqref{z2}. In contrast to the case of the continuity equation, we cannot achieve the result in the whole domain $\Td$, as \eqref{C_W(L^q)} holds only in $Q_F$. However, we can avoid this problem by choosing test functions of class $\bfphi \in C^1_c([0,T]\times Q_F;\R^3)$, thus eliminating the penalization term, as its compact support is considered only in $\Td \setminus Q_F$. We now pass to the limit. Thanks to \eqref{SC1}, \eqref{C_W(L^gamma)}, \eqref{C_W(L^q)}, and Lemma \ref{konvergence-magnet.field} (to handle the Lorentz force term), and based on the assumptions about the initial data, we may pass to the limit in most terms of the momentum equation \eqref{z2}. In particular, for the convective term, note that it holds by Hölder's inequality
\begin{align*}
    \int_0^\tau \|\vm_n \otimes \vun \|^2_{L^q(\Td;\R^{3\times3})} \dt &\leq  \int_0^T \|\vm_n \|^2_{L^\frac{2\gamma}{\gamma+1}(\Td;\R^3)} \| \vun\|^2_{L^6(\Td;\R^3)}\dt \br
    &\leq \left( \sup_{t\in[0,T]}\|\vm_n \|^2_{L^\frac{2\gamma}{\gamma+1}(\Td;\R^3)}\right) \int_0^T \| \vun\|^2_{W^{1,2}(\Td;\R^3)}\dt,
\end{align*}
for
\begin{equation*}
    q=\frac{6\gamma}{4\gamma+3}=\left(\frac{1}{6}+\frac{\gamma+1}{2\gamma} \right)^{-1}.
\end{equation*}
This implies that
\begin{equation*}
    \vm_n \otimes \vun \to \Ov{\vm \otimes \vu} \ \text{ weakly in } L^2(0,T;L^q(\Td;\R^{3\times3})).
\end{equation*}
Passing to the suitable subsequence, we obtained the limit of \eqref{z2}  
\begin{align}
	\left[ \intTd{ \vm \cdot \bfphi } \right]_{t = 0}^{t = \tau} &= \int_0^\tau \intTd{ \Big[ \vm \cdot \partial_t \bfphi + 
		\Ov{\vm \otimes \vu} :\Grad \bfphi + \Ov{p(\vr)} \Div \bfphi \Big] } \dt \br &- \int_0^\tau \intTd{ \mathbb{S} : \Grad \bfphi } \dt + 
	\int_0^\tau \intTd{ {\Curl \vc{B} \times \vc{B}} \cdot \bfphi } \dt \br 
    &+\int_0^\tau \intTd{ {\vr \Grad G \cdot \bfphi }} \br
	\vm(0, \cdot) &= \vm_0
	\label{F4}
\end{align}	
for any $\bfphi \in C^1_c([0,T] \times Q_F;\R^3)$ $\prst$-a.s.. 

Finally, fixing $\omega \in \Omega$ and applying deterministic arguments, we establish
\begin{align} 
\vm = \vr \vu, \ \Ov{ \vm \otimes \vu} = \vr \vu \otimes \vu, 
\label{F5}
\end{align}	 
and 
\begin{equation} \label{F6}
	\Ov{p(\vr)} = p(\vr) \ \ \mbox{in}\ (0,T) \times K 
\ \mbox{for any compact}\ K \subset Q_F,
\end{equation}
where \eqref{F6} follows from \emph{the oscillation defect measure} argument  that holds for $\gamma>\frac{3}{2}$ (see, e.g., \cite{FePokTryKar}). This means that the weak limits are unique; therefore, there was no need to extract a subsequence. The renormalized continuity equation \eqref{w2re} is proven in the same paper \cite{FePokTryKar}.

\subsection{Energy inequality}
\label{07-energ-ineq}
To complete the proof, it only remains to pass to the limit in the energy inequality
\begin{align}
    &\left[ \int_{\Td} E(\vrn, \vrn \vun)  \dx + 
\frac{1}{2}\intTd{ |\vc{B}_n|^2  } \right]^{t=\tau}_{t=0} + \int_0^\tau \intTd{ \beta _n |\vun|^2 } \dt  \br&+ 
			\int_0^\tau \intTd{ \mathbb{S}(\Grad \vun) : \Grad \vun }\dt
			+ \int_0^\tau \intTd{ \eta_n |\Curl \vc{B}_n |^2 } \dt 
			\br 
			&\leq\frac{1}{2} \left(\int_0^\tau \intTd{ \left[ \alpha^2 |\vc{B}_n|^2(t, \cdot) +  |\Curl \vc{j}|^2\right] } \dt 
\right)+\int_0^\tau \intTd{\vrn \vun \cdot \Grad G}\dt\br &+ \int_0^\tau \left( \intTd{ \alpha |\vc{B}_n|^2(t, \cdot) } \right) \D W^1_n + \int_0^\tau \left( \intTd{  \Curl \vc{j} \cdot 
\vc{B}_n(t, \cdot)} \right) \D W^2_n.
\label{final_energ}
\end{align}

Given the convergences proved above, we can pass to the limit in \eqref{final_energ} only by means of the weak lower semi--continuity of the energy and the convergences acquired above. For example, it holds $\prst$-a.s.
\begin{align*}
    \int_0^\tau \intTd{\left[\mu |\Grad \vu|^2+(\mu + \nu) |\Div \vu|^2 \right]}\dt&=\int_0^\tau \intTd{\mathbb{S}(\Grad \vu): \Grad \vu}\dt \br
    &\leq \liminf_{n\to \infty} \int_0^\tau \intTd{\mathbb{S}(\Grad \vun): \Grad \vun}\dt
\end{align*}
and similarly
\begin{align*}
    \int_0^\tau \intTd{ \eta |\Curl \vc{B} |^2 } \dt\leq \liminf_{n\to \infty} \int_0^\tau \intTd{ \eta_n |\Curl \vc{B}_n |^2 } \dt.
\end{align*}

Convergence of the gravitational term follows from the strong convergence of $\vrn$ and the weak convergence of $\vun$. Furthermore, by the strong convergence of $\vc{B}_n$,
\begin{equation*}
    \int_0^\tau \intTd{ |\vc{B}_n|^2(t, \cdot) } \dt  \rightarrow \int_0^\tau \intTd{ |\vc{B}|^2(t, \cdot) } \dt.
\end{equation*}

The initial conditions were chosen so that
\begin{equation*}
    \lim_{n\to \infty} \int_{\Td} \left[ E(\vrn, \vrn \vun)(0,\cdot)  + 
\frac{1}{2}|\vc{B}_n(0,\cdot)|^2\right] \dx = \int_{Q_F} E(\vr_0, \vm_0)  \dx + 
\frac{1}{2}\intTd{ |\vc{B}_0|^2}.
\end{equation*}

For the remaining non-stochastic term, the weak lower semi-continuity and strict convexity of $E(x,\vc{y})$ yield
\begin{align*}
    &\int_{Q_F} E(\vr, \vr \vu)(\tau,\cdot) \, \dx + \frac{1}{2}\intTd{|\vc{B}(\tau,\cdot)|^2}\br
    & \leq \intTd{\left[ E(\vr, \vr \vu)(\tau,\cdot)+\frac{1}{2}|\vc{B}(\tau,\cdot)|^2\right]}\br
    &\leq \lim_{n\to \infty} \int_{\Td} \left[ E(\vrn, \vrn \vun)(\tau,\cdot)  + 
    \frac{1}{2}|\vc{B}_n(\tau,\cdot)|^2\right] \dx,
\end{align*}
where in the first inequality, we bounded the integral from above by the fact that the term $E(\vr, \vr \vu)$ is nonnegative.

It remains to address the stochastic terms. We begin with the simpler term, showing that $\prst$-a.s.
\begin{equation*}
    \int_0^\tau \left( \intTd{  \Curl \vc{j}_n \cdot \vc{B}_n(t, \cdot)} \right) \D W^2_n \to \int_0^\tau \left( \intTd{  \Curl \vc{j} \cdot 
\vc{B}(t, \cdot)} \right) \D W^2.
\end{equation*}

To apply Lemma \ref{limit-stoch integ}, it suffices to show convergence in probability. Indeed
\begin{align*}
    \mathbb{E}&\left| \int_0^T \intTd{\Curl \vc{j}_n \cdot (\vc{B}_n-\vc{B})}\dt\right|\br
    & \leq \expe{\| \Curl \vc{j}_n\|_{L^2((0,T)\times \Td;\R^3)} \| \vc{B}_n - \vc{B}\|_{L^2((0,T)\times \Td;\R^3)}} \br
    & \leq \left( \mathbb{E}\| \Curl \vc{j}_n\|^2_{L^2((0,T)\times \Td;\R^3)}\right)^{\frac{1}{2}} \left(\mathbb{E}\| \vc{B}_n - \vc{B}\|^2_{L^2((0,T)\times \Td;\R^3)} \right)^{\frac{1}{2}},
\end{align*}
which gives the result.

For the second stochastic integral term, we need to show that
\begin{equation}\label{4z}
    \intTd{ | \vc{B}_n |^2} \to \intTd{ | \vc{B} |^2} \ \text{ strongly in }  L^2(0,T) \text{ in probability}.
\end{equation}

We know from Lemma \ref{konvergence-magnet.field} that $\prst$-a.s.
\begin{align*}
    \vc{B}_n &\to \vc{B} \ \text{ strongly in } L^2((0,T)\times \Td;\R^3),\br
    \vc{B}_n &\to \vc{B} \ \text{ weakly-(*) in } L^\infty(0,T;L^2(\Td;\R^3).
\end{align*}

By interpolation, we obtain
\begin{equation*}
    \|f\|_{L^q(0,T)}\leq \| f \|^\alpha_{L^2(0,T)} \|f \|^{1-\alpha}_{L^\infty(0,T)}, \ \text{for } q = \frac{2}{\alpha}, \alpha\in [0,1].
\end{equation*}
In particular, by the choice
\begin{equation*}
    f=\| \vc{B}_n - \vc{B}\|_{L^2(\Td;\R^3)},
\end{equation*}
we get the strong convergence in $L^q(0,T)$ for $1\leq q<\infty$, which implies \eqref{4z}. Thus, after applying Lemma \ref{limit-stoch integ}, we get $\prst$-a.s.
\begin{equation*}
    \int_0^\tau \left( \intTd{\alpha |\vc{B}_n(t,\cdot)|^2}\right)\D W^1_n \to \int_0^\tau \left( \intTd{\alpha |\vc{B}(t,\cdot)|^2}\right)\D W^1.
\end{equation*}
Finally, combining the convergence results obtained above, we may pass to the limit in the energy inequality and recover \eqref{w6}. Consequently, the limit object
\[
\mathcal{U} = \Big( (\Omega, \mathcal{A}, (\mathcal{A}_t)_{t \geq 0}, 
\prst ) , \vr, \vu, \vc{B}, W^1, W^2 \Big), 
\] 	
satisfies all conditions of Definition \ref{wD1}. This completes the proof of Theorem \ref{MT1}.
\\
\\
Acknowledgment: The author would like to thank his supervisor, Prof.~Eduard Feireisl, for valuable discussions and guidance, without whom this work would not have been completed.

\def\bibfont{\hfuzz=2pt}
\printbibliography[heading=bibintoc]

@book{FeireislEduard2017SLiT,
author = {Feireisl, Eduard and Novotný, Antonín},
address = {Cham},
copyright = {Springer International Publishing AG 2017},
edition = {2nd ed. 2017 edition.},
isbn = {3319637819},
issn = {2297-0320},
language = {eng},
publisher = {Springer Nature},
series = {Advances in Mathematical Fluid Mechanics},
title = {Singular Limits in Thermodynamics of Viscous Fluids},
year = {2017},
}

@book{FeireislEduard2003DoVC,
author = {Feireisl, Eduard},
address = {Oxford},
isbn = {0198528388},
language = {eng},
publisher = {Oxford University Press},
series = {Oxford Lecture Series in Mathematics and Its Applications},
title = {Dynamics of Viscous Compressible Fluids},
year = {2003},
}

@unknown{WangH21,
author = {Huaqiao Wang},
year = {2021},
month = {08},
pages = {},
title = {Martingale solutions for the compressible MHD systems with stochastic external forces},
doi = {10.48550/arXiv.2108.03844}
}

@article{WangH,
  author       = {Huaqiao Wang},
  title        = {Low Mach number limit of solutions to the stochastic compressible magnetohydrodynamic equations},
  journal      = {Journal of Dynamics and Differential Equations},
  volume       = {35},
  number       = {3},
  pages        = {2413--2451},
  year         = {2023}
}

@article{Lassner,
  author    = {Gerd Lassner},
  title     = {Über ein Rand-Anfangswertproblem der Magnetohydrodynamik},
  journal   = {Archive for Rational Mechanics and Analysis},
  volume    = {25},
  number    = {5},
  pages     = {388--405},
  year      = {1967},
  month     = jan,
  doi       = {10.1007/BF00291938},
  publisher = {Springer}
}

@article{BaFeLMMiYu,
  author    = {Basa\v{r}i\'{c}, Danica
               and Feireisl, Eduard
               and Luk\'{a}\v{c}ov\'{a}-Medvid\textquotesingle{}ov\'{a}, M\'{a}ria
               and Mizerov\'{a}, Hana
               and Yuan, Yuhuan},
  title     = {Penalization method for the {N}avier--{S}tokes--{F}ourier system},
  journal   = {ESAIM: Mathematical Modelling and Numerical Analysis},
  year      = {2022},
  volume    = {56},
  number    = {6},
  pages     = {1911--1938},
  issn      = {2822-7840},
}

@book{BrFeHobook,
author = {Breit, Dominic and Feireisl, Eduard and Hofmanová, Martina},
address = {Berlin},
booktitle = {Stochastically forced compressible fluid flows},
isbn = {978-3-11-049050-3},
language = {eng},
publisher = {De Gruyter},
series = {De Gruyter series in applied and numerical mathematics, volume 3},
title = {Stochastically forced compressible fluid flows },
year = {2018},
}

@article{FeNeSt,
author = {Feireisl, E. and Neustupa, J. and Stebel, J.},
issn = {0022-0396},
journal = {Journal of Differential Equations},
language = {eng},
number = {1},
pages = {596-606},
title = {Convergence of a Brinkman-type penalization for compressible fluid flows},
volume = {250},
year = {2011},
}

@book{FePokTryKar,
author = {Feireisl, Eduard and Karper, Trygve G and Pokorný, Milan},
address = {Cham},
copyright = {Springer International Publishing Switzerland 2016},
edition = {1st ed. 2016 edition.},
isbn = {3319448358},
issn = {2297-0320},
language = {eng},
publisher = {Springer Nature},
series = {Advances in Mathematical Fluid Mechanics},
title = {Mathematical Theory of Compressible Viscous Fluids: Analysis and Numerics},
year = {2016},
}

@book{billing,
  address = {New York},
  author = {Billingsley, Patrick},
  description = {q-paper},
  edition = {Second},
  isbn = {0-471-19745-9},
  mrclass = {60B10 (28A33 60F17)},
  mrnumber = {MR1700749 (2000e:60008)},
  note = {A Wiley-Interscience Publication},
  publisher = {John Wiley \& Sons Inc.},
  series = {Wiley Series in Probability and Statistics: Probability and
              Statistics},
  title = {Convergence of probability measures},
  year = 1999
}

@article{jakubowski,
author = {Jakubowski, Adam},
year = {1997},
month = {07},
pages = {},
title = {The a.s. Skorohod representation for subsequences in nonmetric spaces},
volume = {42},
journal = {Teoriya Veroyatnostej i Ee Primeneniya}
}

@book{Battaner1996,
  author    = {Battaner, E.},
  title     = {Astrophysical Fluid Dynamics},
  publisher = {Cambridge University Press},
  address   = {Cambridge},
  year      = {1996},
  isbn      = {9780521437479},
  url       = {https://www.cambridge.org/9780521437479},
}

@article{2002DomainDependence,
  author    = {Feireisl, Eduard and Novotn{\'y}, Anton{\'i}n and Petzeltov{\'a}, Hana},
  title     = {On the domain dependence of solutions to the compressible
               {N}avier--{S}tokes equations of a barotropic fluid},
  journal   = {Mathematical Methods in the Applied Sciences},
  volume    = {25},
  number    = {12},
  pages     = {1045--1073},
  year      = {2002},
  publisher = {John Wiley \& Sons},
  doi       = {10.1002/mma.333},
}

@book{davidson2001mhd,
  title={An Introduction to Magnetohydrodynamics},
  author={Davidson, P. A.},
  year={2001},
  publisher={Cambridge University Press}
}

@book{taylor1997error,
  title={An Introduction to Error Analysis},
  author={Taylor, John R.},
  edition={2},
  year={1997},
  publisher={University Science Books}
}

@book{clayton1983stellar,
  title={Principles of Stellar Evolution and Nucleosynthesis},
  author={Clayton, Donald D.},
  year={1983},
  publisher={University of Chicago Press}
}

\end{document}